%% file: main_with_erratum.tex
\documentclass{siamart190516}

\input{ex_shared}

\ifpdf
\hypersetup{
  pdftitle={Lipschitz stability for Backward Heat Equation with application to Fluorescence Microscopy},
  pdfauthor={P. Arratia, M. Courdurier, E. Cueva, A. Osses, B. Palacios}
}
\fi

\begin{document}

\maketitle

% REQUIRED
\begin{abstract}
This second version of the manuscript includes, in the appendices, an erratum that points out an error on the published version and offers alternative results for the Lipschitz stability analysis of the backward heat propagation problem and its applications to light sheet fluorescence microscopy.\\

{\em Abstract of the first version.} In this work we study a Lipschitz stability result in the reconstruction of a compactly supported initial temperature for the heat equation in $\R^n$, from measurements along a positive time interval and over an open set containing its support. We take advantage of the explicit dependency of solutions to the heat equation with respect to the initial condition. By means of Carleman estimates we obtain an analogous result for the case when the observation is made along an exterior region  $\omega\times(\tau,T)$, such that the unobserved part $\R^n\backslash\omega$ is bounded. In the latter setting, the method of Carleman estimates gives a general conditional logarithmic stability result when initial temperatures belong to a certain admissible set, and without the assumption of compactness of support. Furthermore, we apply these results to deduce a similar result for the heat equation in $\R$ for measurements available on a curve contained in $\R\times [0,\infty)$, from where a stability estimate for an inverse problem arising in 2D Fluorescence Microscopy is deduced as well. In order to further understand this Lipschitz stability, in particular, the magnitude of its stability constant with respect to the noise level of the measurements, a numerical reconstruction is presented based on the construction of a linear system 
%  in the application of the result to the setting of 
for the inverse problem in 
Fluorescence Microscopy. We investigate the stability constant with the condition number of the corresponding matrix.\\

{\em Abstract of the erratum.} The purpose of the erratum (Appendix \ref{erratum}) is to point out an error in theorem 1.3 which results in the invalidity of its conclusion. In addition, we include an example showing that theorem 1.3 cannot hold and propose alternative (although weaker) results. More specifically, a Lipschitz stability in an $L^1$-norm is replaced by one involving an $H^{-2}$-norm. We also revise other results which were obtained from theorem 1.3, in particular, we modify the stability theorem for the inverse problem in Light Sheet Fluorescence Microscopy which is now a logarithmic one in the natural Lebesgue spaces or Lipschitz in the weaker $H^{-2}$-norm.

\end{abstract}

% REQUIRED

% REQUIRED

\section{Introduction}\label{sec:intro}

In this paper we consider the heat equation in $\R^n$:
\begin{equation}\label{eq:Heat Equation}
\left\{ \begin{array}{rcll} 
u_t-\Delta u & = & 0 & \text{in } \mathbb{R}^n\times(0,T),\\
u(y,0) & = & u_0(y) & \text{in } \mathbb{R}^n,\\
\displaystyle\lim_{|y|\to\infty}u(y,t) & = & 0 & t\in (0,T).   
\end{array}\right.
\end{equation}
We are interested in the reconstruction of the initial temperature $u_0$ when measurements are available in a certain open region. This problem is known as the backward heat equation inverse problem and is an ill-posed problem in the sense of Hadamard \cite{Hadamard_Morse_1953}, \emph{i.e.}, small noise on observations may cause large errors in the reconstruction of the initial condition. Ill-posedness may be overcome by incorporating a priori information about the solutions. A common hypothesis that frequently appears in the literature consists in assuming that the initial condition belongs to a bounded set of some Sobolev space \cite{Fu_Xiong_Qian_2007, Hao_Duc_2009, Yamamoto, vo2017local, Xu_Yamamoto_2000}. This approach is taken into account in order to deduce a conditional logarithmic stability when measurements are made on $\omega \times (\tau,T)$, for $0\leq \tau<T$ and $\omega$ an open set with bounded complement. Namely, given $\beta>0$ and $M>0$, we consider the following admissible set:

\begin{equation}\label{def:Admissible set}
\mathcal{A}_{\beta,M}:=\{a \in H^{2\beta}(\R^n): ||a||_{H^{2\beta}(\R^n)}\leq M\}.
\end{equation}

The conditional logarithmic stability is stated as follows: 

\begin{theorem}\label{th:Log-stability}
Let $u$ be a solution of \cref{eq:Heat Equation} with $u_0\in \mathcal{A}_{\beta,M}$. Let $\omega \times (\tau,T)$ be the observation region where $0\leq\tau<T$ and $\omega\subseteq \R^n$ is an open set such that $\R^n \backslash \omega$ is compact. Let us suppose that  $||u||_{L^2(\omega\times(\tau,T))}<1$. Then, there exist constants $\kappa \in (0,1)$ and $C_1 = C_1(M,\beta,\tau, T, \omega)>0$ such that
\[||u_0||_{L^2(\R^n)}\leq C_1(-\log||u||_{L^2(\omega\times(\tau,T))})^{-\kappa}.\]
\end{theorem}

%%%REVISAR ESTE PARRAFO
%To conclude this result, the most critical step is given by theorem \ref{th:Energy}, where a Carleman inequality for parabolic equations obtained in \cite{Imanuvilov} is needed. Since this inequality is not valid for unbounded domains we base on \cite{Cabanillas} to fix this: null controllability for the heat equation is shown there for control regions with bounded complement. Such result fails when that hypothesis is not considered, see for instance \cite{Lack1, Micu_onthe}. 
To conclude this result, we use a Carleman inequality obtained in \cite{Imanuvilov}. The main problem is that the inequality established in \cite{Imanuvilov} does not hold for unbounded domains such as $\R^n$. In order to be able to apply the Carleman estimate, we use some ideas taken from \cite{Cabanillas}, where null controllability for the heat equation is proved for a control region with bounded complement. Such a large region of control seems to be necessary, as shown in \cite{Lack1, Micu_Zuazua}. %This justifies that the unobserved region $\R^n \backslash \omega$ must be bounded.

The main result of this paper is in the context of compactly supported initial conditions, where the above logarithmic inequality can be improved to a Lipschitz one. The precise result is stated in \cref{th:Lipschitz stability unbounded domain} below. It is in fact a consequence of an analogous result, \cref{th:Lipschitz stability}, for the closely related inverse problem of backward heat propagation with observation in an open region surrounding the support of the initial heat profile.

\begin{theorem}\label{th:Lipschitz stability unbounded domain}
For $R>0$ we define $B:=B(0,R)$ the ball of radius $R$ and centered at the origin. Let $0\leq\tau<T$ and $\omega\subseteq\R^n$ be such that $\R^n\backslash\omega$ is compact and $B\subseteq \R^n\backslash \omega$. Let $u_0\in L^1(\R^n)$ be with $\supp(u_0)\subseteq B$ and $u$ be the respective solution of \cref{eq:Heat Equation}. Then there exists a constant $C_2=C_2(R,\tau,T,\omega)>0$ such that

\[||u_0||_{L^1(\R^n)}\leq C_2||u||_{L^2(\omega\times(\tau,T))}.\]
\end{theorem}

\begin{theorem}\label{th:Lipschitz stability}
Let $B$ be as before. If $u_0\in L^1(\R^n)$ with $\supp(u_0)\subseteq B$, then there exists a constant $C_3 = C_3(R,t_1,t_2)>0$, for $0<t_1<t_2$, such that
\[||u_0||_{L^1(\R^n)}\leq C_3||u||_{L^2(2B\times(t_1,t_2))}.\]
\end{theorem}

\Cref{th:Lipschitz stability} states that we can get an estimate of the initial condition $u_0$ with respect to observations made on an open set containing the support of $u_0$ and for times in a positive interval, while \cref{th:Lipschitz stability unbounded domain} states the analogous result for exterior measurements. To the best of our knowledge, few results about Lipschitz stability for backward heat equation exist in the literature. In \cite{Xu_Yamamoto_2000}, a similar estimate is obtained for the reconstruction of the solution at a positive time $t>0$ and measurements acquired on a subdomain, while in \cite{Saitoh_Yamamoto_1997}, a Lipschitz stability estimate is obtained for the problem of reconstructing the initial condition, although with a very strong norm associated to the (boundary) observations that involve the use of time derivatives of all orders. In our case, we exploit the explicit dependency on the heat equation solution in all of $\R^n$ with respect to the initial condition, as the convolution with the heat kernel.  

The study of the backward heat equation with compactly supported initial conditions arises from an inverse problem related to the microscopy technique performed by a Light Sheet Fluorescence Microscope (LSFM) \cite{Girkin_2018, Huisken_Stainier_2009, Lakowicz_2006}. Images obtained from this kind of microscopes present undesirable properties such as blurring and calibration problems so that, in order to improve the final images, a mathematical direct model was established in \cite{cueva2020mathematical} with the aim of characterizing and analyzing the imaging modality as an inverse problem. Such approach is applied to the imaging of two dimensional specimens, where the light sheet illumination reduces to a laser beam emitted at different heights $y$. The fluorescent distribution is denoted by $\mu$ and is the physical quantity to be reconstructed. At the end of the process, the measurement $p(s,y)$ obtained at pixel $s$ of the camera for the illumination at height $y\in Y_s$ is given by the following expression:
\begin{equation}\label{eq:Measurements Intro}
p(s,y) = c\cdot\text{exp}\left(-\dint_{\gamma(y)}^s\lambda(\tau,y)d\tau\right)\dint_{\R}\dfrac{\mu(s,r)e^{-\int_r^{\infty}a(s,\tau)d\tau}}{\sqrt{4\pi \sigma(s,y)}}\text{exp}\left(-\dfrac{(r-h)^2}{4\sigma(s,y)}\right)dr,
\end{equation}
where 
\[\sigma(s,y) = \dfrac{1}{2}\dint_{\gamma(y)}^s(s-\tau)^2\psi(\tau,y)d\tau.\]
Here, $\lambda, a$ and $\psi$ are physical parameters related with attenuation or scattering and $\gamma$ is a function related with the geometry of $\Omega$, more specifically, $\gamma(y)$ is defined such that $(\gamma(y), y)$ is the first point at height $y$ belonging to $\partial{\Omega}$. These and other terms shall be presented in detail in \cref{sec:LSFM-stability}.

If we fix pixel $s$ and take 
\[u_0(y):=\mu(s,y)e^{-\int_y^{\infty}a(s,\tau)d\tau},\]
the solution $u$ of \cref{eq:Heat Equation} with $n=1$ evaluated in $(y,\sigma(s,y))$ gives us the measurement obtained by the camera at the pixel $s$ for an illumination made at height $y$. Furthermore, $\mu$ is compactly supported, hence $\mu(s,\cdot)$ is as well. The relation between measurements and $u$ is given by the next expression:
\[\begin{array}{rl}
&p(s,y)=c\cdot\text{exp}\left(-\dint_{\gamma(y)}^s\lambda(\tau,y)d\tau\right)u(y,\sigma(s,y))\\ 
\Longleftrightarrow& u(y,\sigma(s,y)) = \dfrac{1}{c} \text{exp}\left(\dint_{\gamma(y)}^s\lambda(\tau,y)d\tau\right) p(s,y).\end{array}\]
This tells us that if we know physical the parameters $\lambda, \psi$ and $a$, then we have access to measurements of $u$ along the curve $\Gamma = \{(y,\sigma(s,y)):y\in Y_s\}\subseteq \R \times (0,T)$. Consequently, the inverse problem consists in the recovery of the initial temperature from these observations.

Uniqueness has been proved in \cite{cueva2020mathematical} based on classical unique continuation results for parabolic equations.  In this paper we also study the Lipschitz stability in the reconstruction of the fluorescence source $\mu$ from measurements available on $\Gamma$. This result will be a direct consequence of the following theorem for the reconstruction of the initial temperature from observations made on a curve contained in $\R\times[0,\infty)$, which is constructed as the graph of a function $\sigma$ that satisfies the following $\sigma$-\emph{properties}:
\begin{enumerate}[i)]
\item $\sigma \in C^1(\R)$,
\item $\sigma>0$ for $y\in (a_1,a_2)$ and $\sigma(y)\equiv 0 $ for $y\in (a_1,a_2)^c$, for some $a_1<a_2$,
\item there exists $\xi_1,\xi_2>0$ such that $\sigma'>0$ in $(a_1,a_1+\xi_1]$, $\sigma'<0$ in $[a_2-\xi_2,a_2)$ and $\sigma(a_1+\xi_1)=\sigma(a_2-\xi_2)$,
\item $\dfrac{1}{\sigma'(y)} = \mathcal{O}\left(\exp\left(\dfrac{1}{\sigma(y)}\right)\right)$ as $y$ goes to $a_1^+,a_2^-$.
%\item there exists constants $c_1,c_2>0$ such that $\sigma'(y)\geq c_1\exp\left(-\dfrac{1}{\sigma(y)}\right)$ for $y\in (a_1,a_1+\xi_1)$.
%\item there exists a constant
%\item $\displaystyle\lim_{y\to a_1,a_2}\dfrac{1}{\sigma(y)^{1/2}\sigma'(y)}\exp\left(-\dfrac{1}{\sigma(y)}\right)=0$.
\end{enumerate}
The theorem is stated as follows:
\begin{theorem}\label{th:Curve stability}
Consider $\sigma:\R\to\R_+$ a function satisfying the  $\sigma$-properties. Let $u$ be the solution of \cref{eq:Heat Equation} with $n=1$ for some $u_0\in L^1(\R)$ such that $\supp(u_0)\subset (a_1+\delta, a_2-\delta)$, where $0<\delta<(a_2-a_1)/2$. Let $\Gamma_L:=\{(y,\sigma(y)):y\in(-\infty,a_1+\xi_1]\}$ and $\Gamma_R:=\{(y,\sigma(y)):y\in[a_2-\xi_2,\infty)\}$ be two curves contained in $\R\times[0,\infty)$ where measurements are available. Then there exists a constant $C_4 = C_4(\sigma,\delta)>0$ such that
\[||u_0||_{L^1(\R)}\leq C_4||u||_{L^1(\Gamma_L\cup\Gamma_R)}\]
\end{theorem}

\begin{remark}
The mentioned $\sigma$-properties, specially the last one, may not be necessary conditions to conclude \cref{th:Curve stability}, but are suitable for the LSFM inverse problem.
\end{remark}

In particular, this theorem implies uniqueness for the inverse problem. Numerical results are carried out after discretizing \cref{eq:Measurements Intro}. Notice that measurements are linear with respect to $\mu$, hence we investigate the stability of the LSFM problem by solving a linear system. Moreover, for the matrix associated, we study its condition number in order to appreciate the behavior of the stability constant. At this point we have to be careful: a Lipschitz type stability may be good from the mathematical point of view, but if the constant is too large with respect to noise level measurements, then the numerical reconstruction may not be satisfactory. Finally, we consider what happens with the reconstruction when the physical parameters $\lambda, a$ and $\psi$ depend on $\mu$, \emph{i.e.}, are also unknown.

The paper is organized as follows: \cref{sec:Log-stability} is devoted to demonstrate \cref{th:Log-stability}. There, we introduce \cref{th:Energy} to show an energy estimate of $u$ which is used later in \cref{sec:Lipschitz-stability} to prove \cref{th:Lipschitz stability}. In \cref{sec:Curve Stability} we prove \cref{th:Curve stability}. Section \ref{sec:LSFM-stability} proves the stability of the 2D LSFM inverse problem. Finally, section \ref{sec:Numerical results} studies from the numerical point of view the result obtained for the LSFM problem.

\section{Conditional Logarithmic Stability}\label{sec:Log-stability}

As said before, since the backward heat equation inverse problem is a well known ill-posed problem, we use the admissible set $\mathcal{A}_{\beta,M}$ previously defined in \cref{def:Admissible set} to add some a priori information on the solution.

To prove \cref{th:Log-stability} let us demonstrate two theorems:

\begin{theorem}\label{th:Energy}
Let $0\leq\tau<T$ and $\omega\subseteq\R^n$ be an open set such that $\R^n\backslash\omega$ is compact. Let $u$ be a solution of \cref{eq:Heat Equation}. Then for all $0<\varepsilon<(T-\tau)/2$ there exists a constant $C_5=C_5(\varepsilon,\tau,T,\omega)>0$ such that
\[||u_t||_{L^2(\tau+\varepsilon,T-\varepsilon;H^{-1}(\R^n))}+||u||_{L^2(\tau+\varepsilon,T-\varepsilon;H^{1}(\R^n))}\leq C_5||u||_{L^2(\tau,T;L^2(\omega))}.\]  
\end{theorem} 

\begin{remark}
The above result holds true even for $\tau=0$ but the constant $C_5$ tends to $\infty$ as $\varepsilon$ tends to $0$.
\end{remark}

\begin{remark}
\Cref{th:Energy} holds true also after replacing $\R^n$ by an unbounded domain $\Omega$ of class $C^2$ uniformly. This could help to extend the logarithmic stability to a more general unbounded set (not only $\R^n$), however, \cref{th:Holder} below fails when dealing with such sets. Hence, the conditional logarithmic stability for a general unbounded set when measurements are made in the region $\omega\times(\tau,T)$ remains an open problem.
\end{remark}

\begin{proof}[Proof of \cref{th:Energy}]

We follow \cite{Cabanillas} to get an estimate of $||u||_{L^2(\tau+\varepsilon,T-\varepsilon;L^2(\R^n))}$. To estimate $||\nabla u||^2_{L^2(\tau+\varepsilon,T-\varepsilon;L^2(\R^n))}$ we make a slight modification to the same argument to conclude. Finally, $||u_t||_{L^2(\tau+\varepsilon,T-\varepsilon;H^{-1}(\R^n))}$ is directly estimated from the heat equation \cref{eq:Heat Equation}. In the points I)-III) presented below, we respectively estimate each one of these terms:

\begin{enumerate}[I)]

\item\label{item:Norm L2 wrt measurements} Estimation of $||u||_{L^2(\tau+\varepsilon,T-\varepsilon;L^2(\R^n))}$. Let $\delta>0$ small enough. Without loss of generality we may assume that $\R^n\backslash\omega$ is connected. Then, we define a cut-off function $\rho \in C^{\infty}(\R^n)$ as follows:% (see \cref{fig:rho}):

\[\left\{\begin{array}{cl}
\rho=1 & \text{ in } \R^n\backslash\omega\\
\rho = 0 & \text{ in } \omega_{\delta}:=\{x\in \omega : d(x,\partial \omega)>\delta\}\\
\rho\in (0,1] & \text{ in } \omega\backslash\omega_{\delta}.
\end{array}\right.\]

% \begin{figure}[h]
% \centering
% \includegraphics[width=12cm]{Figures/rho.pdf}
% \caption{Cut-off function $\rho(x)$ and domain $\Theta$}
% \label{fig:rho}
% \end{figure}

\FloatBarrier

The aim of this function is to localize the solution in the bounded set where observations are not available. Consider $\theta = \rho u$ and let $\Theta=\{x\in \R^n : \rho(x)>0\}$. Notice that $\theta =0 $ in $\R^n\backslash\Theta$ and since $u$ satisfies \cref{eq:Heat Equation}, then $\theta$ satisfies the following parabolic equation in a bounded domain:
\begin{equation}\label{eq:theta}
\left\{\begin{array}{rcll}
\theta_t-\Delta\theta&=&g & \text{ in } \Theta \times (0,T)\\
\theta &=&0& \text{ on } \partial \Theta \times (0,T)\\
\theta(x,0)& =& \rho u_0(x) & \text{ in } \Theta,  
\end{array}\right.
\end{equation}
where $g=-\Delta\rho u - 2\nabla\rho \cdot\nabla u$. Since $\Theta$ is bounded, we can apply the Carleman estimate shown in \cite{Imanuvilov} with $l=1$. More precisely, let $\nu$ be defined as below (we refer to \cite{Chae_Imanuvilov_Kim_1996}, lemma 1.1, for the existence of such function):
\[\left\{\begin{array}{l}
\nu \in C^2(\bar{\Theta})\\
\nu>0 \text{ in } \Theta\text{, }\nu=0 \text{ on } \partial \Theta\\
\nabla\nu \neq 0 \text{ in } \overline{\Theta\backslash\omega} 
\end{array}\right.\]  

and consider the following Carleman weights:
\begin{equation}\label{eq:Carleman weights}\xi(x,t)=\dfrac{e^{\lambda\nu(x)}}{(t-\tau)(T-t)},\quad \zeta(x,t) =\dfrac{e^{\lambda\nu(x)}-e^{2\lambda||\nu(x)||_{C(\overline{\Theta})}}}{(t-\tau)(T-t)}.
\end{equation}
Thus, from \cite{Imanuvilov} we know that the next Carleman estimate holds: there exists $\hat{\lambda}>0$ such that for an arbitrary $\lambda \geq \hat{\lambda}$ there exists $s_0(\lambda)$ and a constant $C>0$ satisfying
\begin{equation}\label{ineq:Carleman inequality}
\begin{array}{l}
\dint_{\tau}^T\dint_{\Theta}\left(\dfrac{1}{s\xi}|\nabla\theta|^2+s\xi|\theta|^2\right)e^{2s\zeta}dxdt\\
\leq C\left(||ge^{s\zeta}||^2_{L^2(\tau,T;H^{-1}(\Theta))}+\dint_{\tau}^T\dint_{\Theta\cap\omega}s\xi|\theta|^2e^{2s\zeta}dxdt\right) \quad \forall s\geq s_0(\lambda),
\end{array}
\end{equation} 
where $\theta$ is the solution of \cref{eq:theta}. Let us estimate the terms in the right and left hand sides of \cref{ineq:Carleman inequality}: 
\begin{itemize}
\item Recall that $g = -\Delta \rho u - 2\nabla\rho \cdot\nabla u$. Noticing that $\Delta \rho = 0$ in $\Theta \backslash \omega$ and $e^{s\zeta}<1$ (since $\zeta<0$), the first term is directly estimated as follows
\[||\Delta\rho u e^{s\zeta}||_{L^2(\tau,T;H^{-1}(\Theta))}^2\leq C||u||^2_{L^2((\tau,T)\times\omega)},\]
where the constant $C>0$ depends on $\rho$. For the second term, we notice that
\[-2(\nabla\rho \cdot \nabla u)e^{s\zeta} = -2\nabla \cdot (ue^{s\zeta}\nabla\rho)+2ue^{s\zeta}\Delta \rho +2use^{s\zeta}\nabla \rho \cdot \nabla \zeta.\]
Again, $\nabla \rho =0$, $\Delta \rho=0$ in $\Theta\backslash \omega$, $e^{s\zeta}<1$. Besides, noticing that there exists $s_1>0$ such that $\left|se^{s\zeta}\dfrac{\partial\zeta}{\partial x_i}\right|<1$, $\forall i\in\{1,\ldots,n\}$, we have that 
\[||2(\nabla\rho \cdot \nabla u)e^{s\zeta}||^2_{L^2(\tau,T;H^{-1}(\Theta))} \leq C ||u||^2_{L^2(\tau,T;L^2(\omega))}, \quad \forall s\geq s_1.\]
The constant $C>0$ depends on $\rho$. Finally, we conclude that
\begin{equation}\label{ineq:Estimation of g} 
||ge^{s\zeta}||_{L^2(\tau,T;H^{-1}(\Theta))}^2\leq C ||u||^2_{L^2((\tau,T)\times\omega)}, \quad \forall s\geq s_1.
\end{equation}

\item We define the functions $\hat{\xi}$, $\hat{\zeta}$ as in \cref{eq:Carleman weights} but with $\lambda = \hat{\lambda}$. Since $\nu \in C^2(\overline{\Theta})$ there exist constants $\eta_1,\eta_2>0$ such that 
\[\dfrac{\eta_1}{(t-\tau)(T-t)}\leq \hat{\xi}\leq \dfrac{\eta_2}{(t-\tau)(T-t)}.\]

Finally, let $\hat{s}:=\displaystyle\max\{s_0(\hat{\lambda}),s_1\}$. Inequality \cref{ineq:Carleman inequality} leads to
\begin{equation}\label{ineq:Carleman inequality 1}
\begin{array}{l}
\dint_{\tau}^T\dint_{\Theta}\left(\dfrac{(t-\tau)(T-t)}{\hat{s}}|\nabla\theta|^2+\dfrac{\hat{s}}{(t-\tau)(T-t)}|\theta|^2\right)e^{2\hat{s}\hat{\zeta}}dxdt\\
\leq C\left(||u||^2_{L^2((\tau,T)\times\omega)}+\dint_{\tau}^T\dint_{\Theta\cap\omega}\dfrac{\hat{s}}{(t-\tau)(T-t)}|\theta|^2e^{2\hat{s}\hat{\zeta}}dxdt\right).
\end{array}
\end{equation}

\end{itemize}

We now estimate the weights in \cref{ineq:Carleman inequality 1} using the following lemma from \cite{Cabanillas}:

\begin{lemma}\label{lem:cotas} 

Let $k,K$ be two positive constants such that

\[k\leq  e^{2\hat{\lambda}||\nu||_{C(\Theta)}}-e^{\hat{\lambda}\nu(x)}\leq K,\quad x\in \overline{\Theta}.\]

Then, for $x\in \overline\Theta$ and $0<\varepsilon<(T-\tau)/2$ we have
\[\left|\left|\dfrac{\hat{s}}{(t-\tau)(T-t)}e^{2\hat{s}\hat{\zeta}}\right|\right|_{L^{\infty}((\Theta\cap\omega)\times(\tau,T))}\leq \dfrac{1}{2k}e^{-1},\]

\[\dfrac{(t-\tau)(T-t)}{\hat{s}}e^{2\hat{s}\hat{\zeta}}\geq \dfrac{\varepsilon(T-\tau-\varepsilon)}{\hat{s}}\text{\emph{exp}}\left(\dfrac{-2\hat{s}K}{\varepsilon(T-\tau-\varepsilon)}\right), \quad t\in\left[\tau+\varepsilon,T-\varepsilon\right].\]

%\emph{Además, si escogemos $\delta>0$ tal que $\tau< \frac{T+3\tau-\delta}{4}$ entonces para $t\in [\frac{T+3\tau-\delta}{4},\frac{3T+\tau}{4}]$}

%\[\dfrac{(t-\tau)(T-t)}{\hat{s}}e^{2\hat{s}\hat{\alpha}}\geq \dfrac{3(T-\tau)^2-\delta^2-2\delta(T-\tau)}{16\hat{s}}\text{exp}\left(\dfrac{-32\hat{s}r(x)}{3(T-\tau)^2-\delta^2-2\delta(T-\tau)}\right)\]

\end{lemma}

%{\bf Observación:} el lado derecho de la tercera desigualdad es menor que el de la segunda. Además, el objetivo de $\delta$ es dar un $\theta_3=\frac{T+3\tau-\delta}{4}$ tal que $\tau<\theta_3<\theta_1<\theta<\theta_2<T$.

%\begin{remark}
%The right-hand side of the last inequality goes to $0$ as $\varepsilon$ goes to $0$.
%\end{remark}

From this lemma, the left-hand side in \cref{ineq:Carleman inequality 1} takes the form
\begin{equation}\label{ineq:Carleman lado izquierdo}
\begin{array}{l}
\hspace{-1cm} \dint_{\tau}^{T}\dint_{\Theta}\left(\dfrac{(t-\tau)(T-t)}{\hat{s}}|\nabla\theta|^2+\dfrac{\hat{s}}{(t-\tau)(T-t)}|\theta|^2\right)e^{2\hat{s}\hat{\zeta}}dxdt \\ 

\hspace{3cm} \geq \dint_{\tau+\varepsilon}^{T-\varepsilon}\dint_{\Theta}\dfrac{(t-\tau)(T-t)}{\hat{s}}|\nabla\theta|^2e^{2\hat{s}\hat{\zeta}}dxdt\\

\hspace{3cm} \geq C(\varepsilon)\dint_{\tau+\varepsilon}^{T-\varepsilon}\dint_{\Theta}|\nabla\theta|^2dxdt,
\end{array}
\end{equation}
where $C(\varepsilon)$ is the lower bound of the last inequality in  \cref{lem:cotas} and tends to $0$ as $\varepsilon$ tends to $0$. On the other hand, the second term in the right-hand side in \cref{ineq:Carleman inequality 1} is estimated as follows:
\begin{equation}\label{ineq:Carleman lado derecho}\begin{array}{rcl}
\dint_{\tau}^T\dint_{\Theta\cap \omega}\frac{\hat{s}}{(t-\tau)(T-t)}|\theta|^2e^{2s\zeta}dxdt&\leq&\dfrac{1}{2k}e^{-1}\dint_{\tau}^T\dint_{\Theta\cap\omega}|\theta|^2dxdt\\

(\text{since }\rho<1 \text{ in } \Theta\cap\omega) &\leq & \dfrac{1}{2k}e^{-1}\dint_{\tau}^T\dint_{\Theta\cap\omega}|u|^2dxdt \\

&\leq&C||u||^2_{L^2(\tau,T;L^2(\omega))}.
\end{array}
\end{equation}

Thus, from \cref{ineq:Carleman inequality 1}, \cref{ineq:Carleman lado izquierdo} and \cref{ineq:Carleman lado derecho} we have that
\begin{equation}\label{ineq:Carleman inequality 2}
\dint_{\tau+\varepsilon}^{T-\varepsilon}\dint_{\Theta}|\nabla\theta|^2dxdt\leq \dfrac{C}{C(\varepsilon)} ||u||^2_{L^2(\tau,T;L^2(\omega))}.
\end{equation}

Since $\theta$ is null on $\partial \Theta$, we use $\lambda_1$ the first eigenvalue of $-\Delta$ in $H_0^1(\Theta)$. Furthermore, $\rho=1$ in $\R^n\backslash\omega$, hence
\[\lambda_1\dint_{\tau+\varepsilon}^{T-\varepsilon}\dint_{\R^n\backslash\omega}|u|^2dxdt = \lambda_1\dint_{\tau+\varepsilon}^{T-\varepsilon}\dint_{\R^n\backslash\omega}|\theta|^2dxdt \leq \dint_{\tau+\varepsilon}^{T-\varepsilon}\dint_{\Theta}|\nabla\theta|^2dxdt.\]

Hence we conclude that
\begin{equation}\label{ineq:Observability u}
||u||^2_{L^2(\tau+\varepsilon,T-\varepsilon;L^2(\R^n))} \leq \dfrac{C}{\lambda_1C(\varepsilon)}||u||^2_{L^2((\tau,T)\times\omega)}.
\end{equation}

\item Estimation of $||\nabla u||_{L^2(\tau+\varepsilon,T-\varepsilon;L^2(\R^n))}$. We focus on the second inequality in \cref{lem:cotas} but for $t\in\left[\tau+\varepsilon/2,T-\varepsilon\right]$. When $t$ is in the latter interval we have the following estimate:
\[\dfrac{(t-\tau)(T-t)}{\hat{s}}e^{2\hat{s}\hat{\zeta}}\geq \dfrac{\varepsilon/2(T-\tau-\varepsilon/2)}{\hat{s}}\text{exp}\left(\dfrac{-2\hat{s}K}{\varepsilon/2(T-\tau-\varepsilon/2)}\right)=:\bar{C}(\varepsilon).\]

Same calculations as in the previous item lead to
\begin{equation}\label{ineq:Carleman inequality 3}
\dint_{\tau+\varepsilon/2}^{T-\varepsilon}\dint_{\R^n}|u|^2dxdt \leq \dfrac{C}{\lambda_1\bar{C}(\varepsilon)}||u||^2_{L^2(\tau,T;L^2(\omega))}.
\end{equation}

Let $\chi(t)\in C^{\infty}([\tau,T])$, with $\chi(t)=0$ for $t\in [\tau,\tau+\varepsilon/2]$, $\chi(t)$ strictly increasing in $(\tau+\varepsilon/2,T-\varepsilon)$, and $\chi(t)=1$ in $[T-\varepsilon,T]$. Multiplying the heat equation \cref{eq:Heat Equation} by $u\chi(t)$ and integrating over $\R^n$, we get %as shown in \cref{fig:chi}:
% \begin{figure}[h]
% \centering
% \includegraphics[width=9cm]{Figures/chi.pdf}
% \caption{Function $\chi(t)$ over the interval $[\tau,T]$}
% \label{fig:chi}
% \end{figure}

\[\dint_{\R^n}|\nabla u|^2\chi dx+\frac{1}{2}\frac{d}{dt}\dint_{\R^n}u^2\chi dx = \frac{1}{2}\dint_{\R^n}u^2\chi_t dx.\]

Now, integrating over $[\tau+\varepsilon/2,t]$:
\[\begin{array}{l}
\dint_{\tau+\varepsilon/2}^{t}\dint_{\R^n}|\nabla u|^2\chi dxdt+\frac{1}{2}\dint_{\R^n}\underbrace{u^2(t)\chi (t)}_{\geq 0}-\underbrace{u^2(\tau+\varepsilon/2)\chi(\tau+\varepsilon/2)}_{=0 \text{, since } \chi(\tau+\varepsilon/2)=0}dx\\
\hspace{5cm}=\dfrac{1}{2}\dint_{\tau+\varepsilon/2}^t\dint_{\R^n}u^2\chi_tdxdt,
\end{array}\]
therefore 
\[\dint_{\tau+\varepsilon/2}^t\dint_{\R^n} |\nabla u|^2\chi dxdt \leq \dfrac{1}{2}\dint_{\tau+\varepsilon}^t\dint_{\R^n}u^2 \chi_t dxdt \leq ||\chi_t||_{\infty}\dint_{\tau+\varepsilon/2}^{T-\varepsilon}\dint_{\R^n}u^2dxdt.
\]

Evaluating at $t=T-\varepsilon$ and using \cref{ineq:Carleman inequality 3} we have 
\begin{equation}\label{ineq:Carleman inequality 4}
\begin{array}{rcl}
\dint_{\tau+\varepsilon/2}^{T-\varepsilon}\dint_{\R^n}|\nabla u|^2\chi dxdt & \leq & ||\chi_t||_{\infty}\dint_{\tau+\varepsilon/2}^{T-\varepsilon}\dint_{\R^n}u^2 dxdt \\

& \leq & \dfrac{C ||\chi_t||_{\infty}}{\lambda_1\bar{C}(\varepsilon)}||u||^2_{L^2(\tau,T;L^2(\omega))}.
\end{array}
\end{equation}

Since $\chi$ is increasing in $(\tau+\varepsilon,T-\varepsilon)$ the left-hand side leads
\begin{equation}\label{ineq:Carleman inequality 5}
\begin{array}{rcl}
\dint_{\tau+\varepsilon/2}^{T-\varepsilon}\dint_{\R^n}|\nabla u|^2\chi dxdt &\geq& \dint_{\tau+\varepsilon}^{T-\varepsilon}\dint_{\R^n}|\nabla u|^2\chi dxdt \\ & \geq & \chi(\tau+\varepsilon)\dint_{\tau+\varepsilon}^{T-\varepsilon}\dint_{\R^n}|\nabla u|^2 dxdt. 
\end{array}
\end{equation}

Bringing \cref{ineq:Carleman inequality 4} and \cref{ineq:Carleman inequality 5} together we get 
\begin{equation}\label{ineq:Observability grad u}
\dint_{\tau+\varepsilon}^{T-\varepsilon}\dint_{\R^n}|\nabla u|^2 dxdt \leq \dfrac{C ||\chi_t||_{\infty}}{\bar{C}(\varepsilon)\chi(\tau+\varepsilon)} ||u||^2_{L^2(\tau,T;L^2(\omega))}.
\end{equation}

Hence, with \cref{ineq:Observability u} and \cref{ineq:Observability grad u} we conclude:  
\[||u||^2_{L^2(\tau+\varepsilon,T-\varepsilon;H^1(\R^n))}\leq C_5||u||^2_{L^2(\omega\times(\tau,T))},\]
where 
\[C_5= \text{exp}\left(\dfrac{2\hat{s}K}{\varepsilon(T-\tau-\varepsilon)}\right) \text{max}\left\{\dfrac{C}{\lambda_1\varepsilon(T-\tau-\varepsilon)},\dfrac{C}{\lambda_1\varepsilon(T-\tau-\varepsilon/2)}\dfrac{||\chi_t||_{\infty}}{\chi(\tau+\varepsilon)}\right\}.\]

\item Estimation of $||u_t||^2_{L^2(\tau+\varepsilon,T-\varepsilon;H^{-1}(\R^n))}$. Multiplying \cref{eq:Heat Equation} by $v\in H^1(\R^n)$ it follows that
\[\dint_{\R^n}u_tv dx=-\dint_{\R^n}\nabla u \cdot\nabla v dx.\]

Integrating over $(\tau+\varepsilon,T-\varepsilon)$ and using the estimate obtained before we conclude
\[||u_t||^2_{L^2(\tau+\varepsilon,T-\varepsilon;H^{-1}(\R^n))}\leq ||u||^2_{L^2(\tau+\varepsilon,T-\varepsilon;H^1(\R^n))}\leq C_5 ||u||^2_{L^2(\omega\times(\tau,T))}.\]

\end{enumerate}

\end{proof}

In the next theorem we assume that the initial condition belongs to the admissible set $\mathcal{A}_{\beta,M}$ defined in \cref{sec:intro}.

\begin{theorem}\label{th:Holder}
Let $0\leq \tau<T$ and $\omega\subseteq\R^n$ be such that $\R^n\backslash\omega$ is compact. Let $u$ be a solution of \cref{eq:Heat Equation} with initial condition $u_0\in \mathcal{A}_{\beta,M}$. Then, for every $\alpha>0$ and $0<\varepsilon<(T-\tau)/2$ there exists a positive constant $C_6=C_6(\alpha,\varepsilon,\tau,T,\omega)$ such that
\[||u||_{C([\tau+\varepsilon,T-\varepsilon];L^2(\R^n))}\leq C_6||u||^{\frac{2\alpha}{2\alpha+1}}_{L^2(\omega\times(\tau,T))}.\]

\end{theorem}

%\begin{remark}
%VEEEER If $\tau>0$ then $C_3\rightarrow\infty$ when $\varepsilon\rightarrow 0$. If $\tau=0$, then $C_4^2/(\tau+\varepsilon)^{2\delta}$ tends to $\infty$ as well. 
%\end{remark}

\begin{remark}
The main consequence of this theorem is a Hölder estimate of the solution $u$ at any time $\tau+\varepsilon\leq t \leq T-\varepsilon$:
\[||u(\cdot,t)||_{L^2(\R^n)}\leq C_6||u||^{\frac{2\alpha}{2\alpha+1}}_{L^2(\omega\times (\tau,T))}.\]
\end{remark}

\begin{proof}

From \cref{th:Energy} there exists a constant $C_5>0$ such that
\[||u||_{H^1(\tau+\varepsilon,T-\varepsilon;H^{-1}(\R^n))}\leq C_5||u||_{L^2(\omega \times (\tau,T))},
\]
and using the Sobolev embedding (see theorem 4.12 in \cite{adams2003sobolev}) we conclude that
\begin{equation}\label{ineq:Sobolev interpolation 1}
||u||_{C([\tau+\varepsilon,T-\varepsilon];H^{-1}(\R^n))}\leq C_5||u||_{L^2(\omega\times (\tau,T))}.
\end{equation}

We now estimate $||u||_{C([\tau+\varepsilon,T-\varepsilon];H^{2\alpha}(\R^n))}$ for some given $\alpha>0$ and conclude the result by interpolation of Sobolev spaces.

Recall that for $a \in H^{2\alpha}(\R^n)$ we have (see e.g. \cite{nezza2011hitchhikers}, proposition 3.4)
\begin{equation}\label{ineq:Laplacian inequality}
\begin{array}{rcl}
||a||_{H^{2\alpha}(\R^n)}^2 &=& ||a||_{L^2(\R^n)}^2+|a|^2_{H^{2\alpha}(\R^n)}\\
&\leq& c(||a||_{L^2(\R^n)}^2+||(-\Delta)^{\alpha}a||^2_{L^2(\R^n)}), 
\end{array} 
\end{equation}
where $|\cdot|_{H^{\alpha}(\R^n)}$ is the seminorm of Gagliardo for fractional Sobolev spaces and $c=c(n,\alpha)$ is a positive constant.

Recall also that $u(\cdot,t) = e^{t\Delta}u_0$, where $e^{t\Delta}$ corresponds to the heat semigroup. Let us estimate the term $||(-\Delta)^{\alpha}u||_{L^2(\R^n)}$ via Fourier. Notice that
\[\mathcal{F}((-\Delta)^{\alpha}e^{t\Delta}u_0) = |\xi|^{2\alpha}e^{-t|\xi|^2}\hat{u}_0.\]%\[\mathcal{F}((-\Delta)^{\alpha}e^{t\Delta}u_0) = \mathcal{F}(\mathcal{F}^{-1}(|\xi|^{2\alpha}e^{-t|\xi|^2}\hat{u})) = |\xi|^{2\alpha}e^{-t|\xi|^2}\hat{u}.\] 
The function $r\in[0,\infty)\to r^{2\alpha}e^{-tr^2}$ reaches its maximum at $\overline{r}=\sqrt{\frac{\alpha}{t}}$ with value $\dfrac{\alpha^{\alpha}e^{-\alpha}}{t^{\alpha}}$, so we conclude that
\begin{equation}\label{ineq:Semigroup}||(-\Delta)^{\alpha}e^{t\Delta}u_0||_{L^2(\R^n)} = |||\xi|^{2\alpha}e^{-t|\xi|^{2}}\hat{u}_0||_{L^2(\R^n)} \leq \dfrac{C(\alpha)}{t^{\alpha}}||u_0||_{L^2(\R^n)}.
\end{equation}
%%%
Bringing \cref{ineq:Laplacian inequality} and \cref{ineq:Semigroup} together, and since $u_0\in \mathcal{A}_{\beta,M}$, we get
\begin{equation}\label{ineq:Sobolev interpolation 2}
\begin{array}{rcl}
||u||_{C([\tau+\varepsilon,T-\varepsilon];H^{2\alpha}(\R^n))}&=& \displaystyle\sup_{t\in [\tau+\varepsilon,T-\varepsilon]}||u(\cdot,t)||_{H^{2\alpha}(\R^n)}\\

%&\leq&
%c\displaystyle\sup_{t\in [\tau+\varepsilon,T-\varepsilon]}\left(||u(\cdot,t)||_{L^2(\R^n)}^2+||(-\Delta)^{\alpha}e^{t\Delta}u_0||_{L^2(\R^n)}^2\right)^{1/2}\\

&\leq & c\displaystyle\sup_{t\in [\tau+\varepsilon,T-\varepsilon]}\left(||u(\cdot,t)||_{L^2(\R^n)}^2+\frac{C^2(\alpha)}{t^{2\alpha}}||u_0||^2_{L^2(\R^n)}\right)^{1/2}\\

&\leq& c\displaystyle\sup_{t\in [\tau+\varepsilon,T-\varepsilon]}\left(||u_0||_{L^2(\R^n)}^2+\frac{C^2(\alpha)}{(\tau+\varepsilon)^{2\alpha}}||u_0||^2_{L^2(\R^n)}\right)^{1/2}\\

&\leq&cM\left(1+\dfrac{C^2(\alpha)}{(\tau+\varepsilon)^{2\alpha}}\right)^{1/2}.
\end{array}
\end{equation}

Finally, we use \cref{ineq:Sobolev interpolation 1} and \cref{ineq:Sobolev interpolation 2}, and conclude via interpolation theory (proposition 2.3 \cite{Lions_Magenes_1972} and section 2.4.1 in \cite{triebel1995interpolation} or theorem 4.1 in \cite{Chandler_Wilde_2014}) taking $s=0, s_0=-1, s_1=2\alpha$ and $\theta = \frac{2\alpha}{2\alpha+1}$ (so that $s = \theta s_0+(1-\theta)s_1$):
\[\begin{array}{rcl}
||u||_{C([\tau+\varepsilon,T-\varepsilon];L^2(\R^n))} &\leq &||u||_{C([\tau+\varepsilon,T-\varepsilon];H^{-1}(\R^n))}^{\frac{2\alpha}{2\alpha+1}}||u||^{\frac{1}{2\alpha+1}}_{C([\tau+\varepsilon,T-\varepsilon];H^{2\alpha}(R^n))}\\

&\leq& \underbrace{\left[c M \left(1+\dfrac{C^2(\alpha)}{(\tau+\varepsilon)^{2\alpha}}\right)^{1/2}\right]^{\frac{1}{2\alpha+1}}C_5^{\frac{2\alpha}{2\alpha+1}}}_{=:C_6}||u||^{\frac{2\alpha}{2\alpha+1}}_{L^2(\omega\times (\tau,T))}.

\end{array}\]

\end{proof}

Let us now derive the conditional logarithmic stability estimate:

\begin{proof}[Proof of \cref{th:Log-stability}]
It suffices to follow steps 2 and 3 in the proof of theorem 2.1 of \cite{Yamamoto}. First of all, the function $t\to||u(\cdot,t)||^2_{L^2(\R^n)}$ is log-convex, then, for $0\leq t\leq \theta$, we note that $t = 0\cdot(1-t/\theta) + \theta\cdot(t/\theta)$ is a convex combination, hence
\[||u(\cdot,t)||^2_{L^2(\R^n)}\leq ||u_0||_{L^2(\R^n)}^{2(1-t/\theta)}||u(\cdot,\theta)||^{2t/\theta}_{L^2(\R^n)}\leq M^{2(1-t/\theta)}||u(\cdot,\theta)||^{2t/\theta}_{L^2(\R^n)}.\] 

Integrating from $0$ to $\theta$ it yields
\[\begin{array}{rcl}
\dint_0^{\theta}||u(\cdot,t)||^2_{L^2(\R^n)}dt&\leq& M^2\dint_0^{\theta}\left(\dfrac{||u(\cdot,\theta)||_{L^2(\R^n)}}{M}\right)^{2t/\theta}dt\\
&=&\theta \left(\dfrac{||u(\cdot,\theta)||^2_{L^2(\R^n)}-M^2}{\log(||u(\cdot,\theta)||^2_{L^2(\R^n)})-\log(M^2)}\right).
\end{array}\]

Due to the logarithm concavity, the right-hand side of the previous estimate is an increasing function with respect to the term $||u(\cdot,\theta)||_{L^2(\R^n)}$, which together with \cref{th:Holder} implies that
\[\dint_0^{\theta}||u(\cdot,t)||^2_{L^2(\R^n)}dt\leq \theta \left(\dfrac{C_6^2||u||_{L^2(\omega \times(\tau,T))}^{\frac{4\alpha}{2\alpha+1}}-M^2}{\log(C^2_6||u||_{L^2(\omega \times(\tau,T))}^{\frac{4\alpha}{2\alpha+1}})-\log(M^2)}\right).\]

Now we have two cases: $C_6\leq M$ or $M\leq C_6$. We study the first case, the second one is analogous. If $C_6\leq M$ then 
\[\dint_0^{\theta}||u(\cdot,t)||^2_{L^2(\R^n)}dt \leq M^2 \theta \left(\dfrac{\frac{C^2_6}{M^2}||u||_{L^2(\omega \times(\tau,T))}^{\frac{4\alpha}{2\alpha+1}}-1}{\log(\frac{C^2_6}{M^2}||u||_{L^2(\omega \times(\tau,T))}^{\frac{4\alpha}{2\alpha+1}})}\right).\]
The right-hand side is increasing as a function of $C_6/M$ and $C_6/M\leq 1$, hence 
\[\dint_0^{\theta}||u(\cdot,t)||^2_{L^2(\R^n)}dt \leq M^2 \theta \left(\dfrac{||u||_{L^2(\omega \times(\tau,T))}^{\frac{4\alpha}{2\alpha+1}}-1}{\log(||u||_{L^2(\omega \times(\tau,T))}^{\frac{4\alpha}{2\alpha+1}})}\right).\]
Since measurements are sufficiently small, \emph{i.e.}, $||u||_{L^2(\omega \times(\tau,T))}<1$, we have that
\[\log||u||^{\frac{4\alpha}{2\alpha+1}}_{L^2(\omega \times(\tau,T))}<0,\]
then
\[\dint_0^{\theta}||u(\cdot,t)||^2_{L^2(\R^n)}dt \leq M^2 \theta \frac{2\alpha+1}{4\alpha}(-\log||u||_{L^2{(\omega\times(\tau,T))}})^{-1}.\]
If $M\leq C_6$, we can follow the same steps obtaining that 
\[\dint_0^{\theta}||u(\cdot,t)||^2_{L^2(\R^n)}dt \leq C^2_6 \theta \frac{2\alpha+1}{4\alpha}(-\log||u||_{L^2{(\omega\times(\tau,T))}})^{-1}\] 
In conclusion, we get the following estimate
\begin{equation}\label{ineq:Step 2}||u||_{L^2(0,\theta; L^2(\R^n))}\leq \displaystyle\max\{C_6,M\}\left(\theta\dfrac{2\alpha+1}{4\alpha}\right)^{1/2}(-\log||u||_{L^2(\omega\times (\tau,T))})^{-1/2}.
\end{equation}

In order to conclude we shall estimate the norms $||u||_{W^{1,p}(0,\theta;L^2(\R^n))}$ and \newline $||u||_{L^{p}(0,\theta;L^2(\R^n))}$ for some $p>1$ and use interpolation of Sobolev spaces and Sobolev embeddings. On one side we have that
\[
u_t(\cdot,t)  = \Delta e^{t\Delta}u_0 = -(-\Delta)^{1-\beta}e^{t\Delta}(-\Delta)^{\beta}u_0 
\]
and thanks to the fractional Laplacian properties (recall \cref{ineq:Semigroup}):
\[||u_t(\cdot,t)||_{L^2(\R^n)} \leq  \dfrac{C(\beta)}{t^{(1-\beta)}}||(-\Delta)^{\beta}u_0||_{L^2(\R^n)}.\]

Let $1<p<1/(1-\beta)$. Using that $u_0\in \mathcal{A}_{\beta,M}$ we get that
\[\begin{array}{rcl}
\dint_0^{\theta}||u_t(\cdot,t)||^p_{L^2(\R^n)}dt & \leq & C(\beta)\dint_0^{\theta}\frac{1}{t^{p(1-\beta)}}dt \hspace{1mm} ||(-\Delta)^{\beta}u_0||^p_{L^2(\R^n)} \\ 
& \leq & C(\beta)\dfrac{\theta^{1-p(1-\beta)}}{1-p(1-\beta)} \hspace{1mm} |u_0|_{H^{2\beta}(\R^n)}^p \\ 
& \leq & C(\beta)\dfrac{\theta^{1-p(1-\beta)}}{1-p(1-\beta)} \hspace{1mm} M^p,
\end{array}\]
that is
\begin{equation}\label{ineq:Sobolev interpolation u_t W1p}||u_t||^p_{L^p(0,\theta;L^2(\R^n))} \leq C(\beta,M,\theta).
\end{equation}

On the other side,
\begin{equation}\label{ineq:Sobolev interpolation u W1p}\dint_0^{\theta}||u(\cdot,t)||^p_{L^2(\R^n)}dt\leq \dint_0^{\theta}||u_0||^p_{L^2(\R^n)}dt \leq M^p\theta.
\end{equation}

Bringing \cref{ineq:Sobolev interpolation u_t W1p} and \cref{ineq:Sobolev interpolation u W1p} together we deduce 
\begin{equation}\label{ineq: Sobolev interpolation 1}
||u||_{W^{1,p}(0,\theta;L^2(\R^n))}\leq C(\beta,M,\theta).
\end{equation}
The previous constant decreases with $\theta$, which means that the stability constant decreases when initial time of observation $\tau$ is closer to $0$. Taking $p\leq 2$, we can use \cref{ineq:Step 2} but with $L^p$ norm in time:
\begin{equation}\label{ineq: Sobolev interoplation 2}
||u||_{L^p(0,\theta;L^2(\R^n))}\leq \theta^{1/p-1/2}||u||_{L^2(0,\theta;L^2(\R^n))}\leq C(-\log||u||_{L^2(\omega\times (\tau,T))})^{-1/2}.
\end{equation}

Again, we interpolate estimates \cref{ineq: Sobolev interpolation 1} and \cref{ineq: Sobolev interoplation 2} so that for $0<s<1$ 
\[||u||_{W^{1-s,p}(0,\theta;L^2(\R^n))}\leq C(-\log||u||_{L^2(\omega\times (\tau,T))})^{-s/2}.\]
Letting $s$ such that $(1-s)p>1$ we can use the Sobolev embedding and conclude with $\kappa = s/2$:\[||u||_{C([0,\theta];L^2(\R^n))}\leq C||u||_{W^{1-s,p}(0,\theta;L^2(\R^n))}\leq C_1 (-\log||u||_{L^2(\omega \times (\tau, T))})^{-s/2}.\]

\end{proof}

\section{Conditional Lipschitz Stability}\label{sec:Lipschitz-stability}

In this section we prove the main results of this paper, \cref{th:Lipschitz stability}, which provides a Lipschitz stability inequality in the recovery of the initial condition when observations are made on some interval $(t_1,t_2)$, with $0<t_1<t_2$, and in an open domain containing the support of the initial condition. \Cref{th:Lipschitz stability unbounded domain} gives a similar conclusion when measurements are made on an unbounded domain that does not necessarily contain the support of the initial condition. This last theorem follows directly from \cref{th:Energy} and \cref{th:Lipschitz stability} and will be used later in section \ref{sec:Curve Stability}. 

To demonstrate \cref{th:Lipschitz stability} let us prove first the following lemma whose main hypothesis is that $u_0\geq0$:

\begin{lemma}\label{lem:u0 Positiva soporte compacto}
If $u_0\in L^1(\R^n)$, $u_0\geq 0$ and $\supp(u_0)\subseteq B:=B(0,R)$, for some $R>0$. Then, for $t>0$ there exists a constant $C_7 = C_7(R,t)>0$ such that

\[||u_0||_{L^1(\R^n)}\leq C_7||u(\cdot,t)||_{L^2(2B)}.\]
\end{lemma}

\begin{proof}
We recall that $u$ takes the explicit form
\begin{equation}\label{eq:Solution u}
u(y,t) = \dint_{\R^n}u_0(r)\dfrac{e^{-\frac{|y-r|^2}{4t}}}{(4 \pi t)^{n/2}}dr.
\end{equation}

Since $u_0\geq 0$ and the heat kernel integrates $1$ for any $t>0$ we have
\begin{equation}\label{eq:Norm L1 u} 
\begin{array}{rcl}
||u_0||_{L^1(\R^n)} = \dint_{\R^n}u_0(r)dr & = & \dint_{\R^n}\dint_{\R^n}u_0(r) \dfrac{e^{-\frac{|y-r|^2}{4t}}}{(4 \pi t)^{n/2}}dr dy \\

& = & \dint_{|y|<2R}\dint_{\R^n}u_0(r)\dfrac{e^{-\frac{|y-r|^2}{4t}}}{(4 \pi t)^{n/2}}drdy \\
&&+ \dint_{|y|>2R}\dint_{\R^n}u_0(r)\dfrac{e^{-\frac{|y-r|^2}{4t}}}{(4 \pi t)^{n/2}}drdy.

\end{array}
\end{equation}

The first integral on the right hand side is easily bounded by Cauchy-Schwarz and recalling \cref{eq:Solution u}:
\begin{equation}\label{ineq:Left hand side norm u}
\dint_{|y|<2R}\dint_{\R^n}u_0(r)\dfrac{e^{-\frac{|y-r|^2}{4t}}}{(4 \pi t)^{n/2}}dr dy = \dint_{|y|<2R}u(y,t)dy \leq |2B|^{1/2}||u(\cdot,t)||_{L^2(2B)},
\end{equation}
where $|2B|$ denotes the volume of the ball of radius $2R$. For the second integral, due to the support of $u_0$ we notice that 
\[\begin{array}{rcl}
\dint_{|y|>2R}\dint_{\R^n}u_0(r)\dfrac{e^{-\frac{|y-r|^2}{4t}}}{(4 \pi t)^{n/2}}drdy &=&\dint_{|y|>2R}\dint_{|r|<R}u_0(r) \dfrac{e^{-\frac{|y-r|^2}{4t}}}{(4 \pi t)^{n/2}}dr dy \\

&=& \dint_{|r|<R}u_0(r)\left(\dint_{|y|>2R}\dfrac{e^{-\frac{|y-r|^2}{4t}}}{(4 \pi t)^{n/2}}dy \right) dr, 
\end{array}\] 
where the integral inside parenthesis can be bounded uniformly with respect to $r$ by a constant $\alpha(R,t)\in(0,1)$, increasing with respect to $t$. This yields
\begin{equation}\label{ineq:Right hand side norm u}
\dint_{|y|>2R}\dint_{\R^n}u_0(r)\dfrac{e^{-\frac{|y-r|^2}{4t}}}{(4 \pi t)^{n/2}}drdy\leq \alpha\dint_{\R^n}u_0(r)dr.
\end{equation}

Bringing \cref{eq:Norm L1 u}, \cref{ineq:Left hand side norm u} and \cref{ineq:Right hand side norm u} together we deduce the estimate:
\[||u_0||_{L^1(\R^n)}=\dint_{\R^n}u_0(r)dr \leq \underbrace{(1-\alpha)^{-1}C_R}_{=:C_7}||u(\cdot,t)||_{L^2(2B)}.\]
\end{proof}

The constant of the previous lemma can be chosen uniformly with respect to $t$ in a closed interval $[t_1,t_2]$ for $t_1>0$:

\begin{corollary}\label{cor:Lipschitz stability bounded domain}
Let $0<t_1<t_2$. There exists a constant $C_8=C_8(R,t_1,t_2)>0$ such that
\[||u_0||_{L^1(\R)}\leq C_8||u||_{L^2(2B\times(t_1,t_2))}.\]
%\[||u_0||_{L^1(\R)}\leq \underbrace{\dfrac{C(R,t_2)}{\sqrt{t_2-t_1}}}_{=:C}||u||_{L^2(2B\times(t_1,t_2))}\]
\end{corollary}

%A direct consequence of this corollary and theorem \cref{th:Holder} is a Hölder stability for this kind of initial conditions:
%
%\begin{corollary}
%Let $0\leq\tau<T$, $0<\varepsilon<(T-\tau)/2$ and let $u_0$ with the same hypothesis as lemma \ref{lem:u0 Positiva soporte compacto} and theorem \ref{th:Holder}. Then there exists a constant $C=C(K,T,\varepsilon,\delta)>0$ such that
%
%\[||u_0||_{L^1(\R)}\leq C||u||_{L^2(\omega\times(\tau,T))}^{\frac{2\delta}{2\delta+1}}\]
%\end{corollary}
%
%\begin{Proof}
%Let $t_1=\tau+\varepsilon$, $t_2=T-\varepsilon$. The previous corollary leads
%
%\begin{equation}\label{ineq:Corolario 2 Corolario 1}||u_0||_{L^1(\R)}\leq \dfrac{C(K,T-\varepsilon)}{\sqrt{T-\tau-2\varepsilon}}||u||_{L^2((-2K,2K),(\tau+\varepsilon,T-\varepsilon))}
%\end{equation}
%
%Theorem \ref{th:Holder} implies 
%\begin{equation}\label{ineq:Corolario 2 Holder}
%||u||_{L^2((-2K,2K),(\tau+\varepsilon,T-\varepsilon))} \leq\left(\dint_{\tau+\varepsilon}^{T-\varepsilon}||u(\cdot,t)||^2_{L^2(\R)}dt\right)^{1/2} \leq \sqrt{T-\tau-2\varepsilon} C_{\varepsilon,\delta}||u||^{\frac{2\delta}{2\delta+1}}_{L^2(\omega\times(\tau,T))}
%\end{equation}
%
%Bringing both expressions together we conclude
%
%\[||u_0||_{L^1(\R)}\leq \dfrac{C(K,T-\varepsilon)}{\sqrt{T-\tau-2\varepsilon}}\sqrt{T-\tau-2\varepsilon} C_{\varepsilon,\delta}||u||^{\frac{2\delta}{2\delta+1}}_{L^2(\omega\times(\tau+\varepsilon,T-\varepsilon))} = \underbrace{C(K,T-\varepsilon)C_{\varepsilon,\delta}}_{C(K,T,\varepsilon,\delta)}||u||^{\frac{2\delta}{2\delta+1}}_{L^2(\omega\times(\tau,T))}\]
%
%\end{Proof}
\noindent What remains to be done is to get rid of the positiveness of $u_0$:

\begin{proof}[Proof of \cref{th:Lipschitz stability}]
Let $u^{\pm}$ be the solution to \cref{eq:Heat Equation} with $u_0^{\pm}=\displaystyle\max\{\pm u_0,0\}$ as initial condition respectively. Noticing that $u_0^{\pm}\geq 0$ and $u^{\pm}\geq 0$, \Cref{cor:Lipschitz stability bounded domain} tells us that there exists a constant $C_8>0$ such that
\[||u_0^{\pm}||_{L^1(\R^n)}\leq C_8||u^{\pm}||_{L^2(2B\times(t_1,t_2))}.\]

Since $u=u^+-u^-$, then 
\begin{equation}\label{ineq:Rango cerrado}
\begin{array}{rcl}
||u_0||_{L^1(\R^n)}&=&||u_0^+||_{L^1(\R^n)}+||u_0^-||_{L^1(\R^n)} \\ 
&\leq& C_8 \left(||u||_{L^2(2B\times(t_1,t_2))}+||u^-||_{L^2(2B\times(t_1,t_2))}\right).
\end{array}
\end{equation}

Let us analyze the two following operators:
\[\Lambda: u_0 \in L^1(B) \to u \in L^2(2B\times(t_1,t_2)) \]
\[\Upsilon: u_0 \in L^1(B) \to u^- \in L^2(2B\times(t_1,t_2)), \]
and prove that $\Lambda$ is a bounded and injective linear operator and $\Upsilon$ is a compact operator. 

In effect, we use Young's inequality with $p=1, q=2$ and $r=2$ (so that $\frac{1}{p}+\frac{1}{q}=\frac{1}{r}+1$) to obtain

\[\begin{array}{rcl}
||u^-(\cdot,t)||_{L^2(2B)}&\leq&||u_0^-||_{L^1(\R^n)}\dfrac{1}{(4\pi t)^{n/2}}||e^{-|y|^2/4t}||_{L^2(\R^n)}\\

&\leq&||u_0^-||_{L^1(\R^n)}\dfrac{1}{(4\pi t)^{n/4}}\left(\dint_{\R^n}\dfrac{1}{(4\pi t)^{n/2}}e^{-|y|^2/4t}dy\right)^{1/2} \\
&\leq&||u_0^-||_{L^1(\R^n)}\dfrac{1}{(4\pi t)^{n/4}},
\end{array}\]
where in the second step we used that $e^{-a/2t}\leq e^{-a/4t}$ for $a>0$. From here we conclude that
\[||u^-||^2_{L^2(2B\times(t_1,t_2))} \leq ||u_0^-||^2_{L^1(\R^n)}\dfrac{1}{(4\pi)^{n/2}}\left\{\begin{array}{ll}\log(t_2/t_1),& \text{ if } n=2\\ \dfrac{1}{n/2-1}\left(\dfrac{t_1}{t_1^{n/2}}-\dfrac{t_2}{t_2^{n/2}}\right),& \text{ if } n\neq 2.\end{array}\right.\]

\begin{comment}
(THE FOLLOWING IS JUST FOR COMPARISON WITH THE PREVIOUS ARGUMENT. IT HAS TO BE ELIMINATED) In effect, by Cauchy-Schwarz we have that
\[\begin{array}{rcl}
u^-(y,t) & = & \dint_{\R^n}u_0^{-}(r)\dfrac{e^{-\frac{|y-r|^2}{4t}}}{(4\pi t)^{n/2}}dr \\
& \leq & ||u_0^-||_{L^1(\R^n)}^{1/2}\left(\dint_{\R^n}u_0^-(r)\dfrac{e^{-\frac{|y-r|^2}{2t}}}{(4\pi t)^n}\right)^{1/2}\\
& \leq & ||u_0^-||_{L^1(\R^n)}^{1/2}\left(\dfrac{u^-(y,t)}{(4\pi t)^{n/2}}\right)^{1/2}, 
\end{array}\]
where in the last step we have used that $e^{-a/2t}\leq e^{-a/4t}$ for $a>0$. From here we conclude that
\[||u^-||^2_{L^2(2B\times(t_1,t_2))} \leq ||u_0^-||^2_{L^1(\R^n)}\dfrac{|2B|}{(4\pi)^n}\left\{\begin{array}{ll}\log(t_2/t_1),& \text{ if } n=1\\ \frac{1}{n-1}\left(\frac{t_1}{t_1^n}-\frac{t_2}{t_2^n}\right),& \text{ if } n> 1.\end{array}\right.\]
\end{comment}

Hence, there exists a constant $C>0$ such that
\begin{equation}\label{ineq: Upsilon continuo}
||u^-||_{L^2(2B\times(t_1,t_2))}\leq C||u_0^-||_{L^1(\R^n)}\leq C||u_0||_{L^1(B)},
\end{equation}
and analogously, we have
\[||u^+||_{L^2(2B\times(t_1,t_2))}\leq  C||u_0||_{L^1(B)}.\]  

Since $u=u^+-u^-$, $\Lambda$ turns out to be a bounded operator:
\[||\Lambda u_0||_{L^2(2B\times(t_1,t_2))}=||u||_{L^2(2B\times(t_1,t_2))}\leq C||u_0||_{L^1(B)}.\] 

Let us verify the compactness of $\Upsilon$. For this purpose we consider $\Upsilon$ as the composition of two operators $\Upsilon=\Upsilon_2 \circ \Upsilon_1$ where 
\[\Upsilon_1: u_0\in L^1(B)\to u^- \in L^2(t_1,t_2;H^1(2B))\]
\[\Upsilon_2: u^-\in L^2(t_1,t_2;H^1(2B)) \to u^- \in L^2(2B\times(t_1,t_2)).\]

We claim that $\Upsilon_1$ is a bounded linear operator while $\Upsilon_2$ is compact. In fact, thanks to \cref{ineq: Upsilon continuo} it suffices to estimate the derivatives in order to conclude the boundedness of $\Upsilon_1$:
%\[\partial_yu^-(y,t)= -\dfrac{1}{2t}\dint_{\R}u_0^-(r)(y-r)\dfrac{e^{-\frac{|y-r|^2}{4t}}}{\sqrt{4\pi t}} dr\]
%
%Hence
%\[|\partial_y u^-(y,t)|\leq \dfrac{1}{2t} \left(\dint_{\R}u_0^-(r)dr\right)^{1/2}\left(\dint_{\R}u_0^-(r)(y-r)^2\dfrac{e^{-\frac{|y-r|^2}{2t}}}{4\pi t}\right)^{1/2}\leq \dfrac{1}{2t}||u_0^-||^{1/2}_{L^1(\R)}\left(\dint_{\R}u_0^-(r)\dfrac{1}{\pi e}dr\right)^{1/2}\]
%
%After some computations we conclude that
%
%\[||\partial_y u^-||_{L^2((-2K,2K)\times(t_1,t_2))} \leq C ||u_0||_{L^1(\R)}\]  
%   
%For the temporal derivative we see that
%
%\[\partial_t u^-(y,t) = \dint_{\R}u_0^-(r)\dfrac{e^{-\frac{|y-r|^2}{4t}}}{\sqrt{4\pi t}}\left(\dfrac{|y-r|^2}{4t^2}-\dfrac{1}{2t}\right)dr\]
%
%As before, easy computations yields
%\[||\partial_y u^-||_{L^2((-2K,2K)\times(t_1,t_2))}\leq C||u_0^-||_{L^1(\R)} 
%\] 
\[\nabla u^-(y,t) = \left(u_0^-(\cdot) * \nabla \dfrac{e^{-|\cdot|^2/4t}}{(4\pi t)^{n/2}}\right)(y) = \left(u_0^-(\cdot) * -\dfrac{\cdot}{2t}\dfrac{1}{(4\pi t)^{n/2}}e^{-|\cdot|^2/4t}\right)(y). \]

To estimate $||\nabla u^-(\cdot,t)||_{(L^2(2B))^n}$ we use Young's inequality with $p, q$ and $r$ as before getting
\[\begin{array}{rcl}

||\nabla u^-(\cdot,t)||_{(L^2(2B))^n} & \leq & ||u_0^-||_{L^1(\R^n)} \dfrac{1}{2(4\pi)^{n/2}t^{n/2+1}}||ye^{-|y|^2/4t}||_{(L^2(\R^n))^n}\\

&=&\dfrac{C}{t^{n/2+1}} ||u_0^-||_{L^1(\R^n)} \left(\dint_{\R^n}|y|^2e^{-|y|^2/2t}dy\right)^{1/2}\\

&=&\dfrac{C}{t^{n/2+1}} ||u_0^-||_{L^1(\R^n)} \left(\dint_0^\infty r^{n+1}e^{-r^2/2t}dr\right)^{1/2}
,
\end{array}\]
where we have used spherical coordinates. Notice that 
\[\dint_0^\infty r^{n+1}e^{-r^2/2t}dr = C(n) t^{n/2+1},\]
hence, 
\[||\nabla u^-(\cdot,t)||_{(L^2(2B))^n}  \leq  C||u_0^-||_{L^1(\R^n)}
\dfrac{1}{t^{n/4+1/2}}. \]
Integrating in time from $t_1$ to $t_2$ we get 
\[||\nabla u^- ||^2_{L^2(2B\times(t_1,t_2))}\leq C||u_0^-||^2_{L^1(\R^n)}\left(\dfrac{1}{t_1^{n/2}}-\dfrac{1}{t_2^{n/2}}\right).\]
Thus we have estimated the spatial derivative
\[||\nabla u^-||_{L^2(2B\times(t_1,t_2))} \leq C ||u_0^-||_{L^1(\R^n)}\leq C ||u_0||_{L^1(B)}.\]  

%For the temporal derivative note that
%\[\partial_t u^-(y,t) = \dint_{\R^n}\dfrac{u_0^-(r)}{(4\pi t)^{n/2}}e^{-\frac{|y-r|^2}{4t}}\left(\dfrac{|y-r|^2}{4t^2}-\dfrac{n}{2t}\right)dr\]
%
%As before, easy computations yields
%\[||\partial_t u^-||^2_{L^2(2B\times(t_1,t_2))}\leq C||u_0^-||^2_{L^1(\R^n)}(n/2+1)\left(\dfrac{1}{t_1^{n/2+1}}-\dfrac{1}{t_2^{n/2+1}}\right)\leq C||u_0||_{L^1(B)}\]

%The constant $C$ depends on $R,t_1,t_2$ and the dimension $n$ and tends to $\infty$ as $t_1\to 0$. 
In conclusion $\Upsilon_1$ is bounded and thanks to Rellich-Kondrachov theorem $\Upsilon_2$ is compact (see for instance theorem 6.3 in \cite{adams2003sobolev}). Consequently, $\Upsilon$ is a compact operator and from \cref{ineq:Rango cerrado} and proposition 6.7 in \cite{Taylor2011} we conclude that $\Lambda$ is a closed operator. Finally, strong unique continuation property of the heat equation implies the injectivity of $\Lambda$, thus, the open mapping theorem gives us the existence of a constant $C>0$ such that
\[||u_0||_{L^1(\R^n)}\leq C ||\Lambda u_0||_{L^2(2B\times(t_1,t_2))}=C||u||_{L^2(2B\times(t_1,t_2))}.\]
\end{proof}

%\begin{remark}
%Note that in the previous estimation we can extend the time interval in the right hand side, ie., if $(t_1,t_2)\subset (\overline{t}_1,\overline{t}_2)$ then
%
%\[||u_0||_{L^1(\R^n)}\leq C||u||_{L^2(2B\times(\overline{t}_1,\overline{t}_2))}\]
%
%\noindent holds for the same constant $C$.
%\end{remark}

We finish this section by demonstrating \cref{th:Lipschitz stability unbounded domain}:

\begin{proof}[Proof of \cref{th:Lipschitz stability unbounded domain}]
Let $t_1 = \tau+\varepsilon$ and $t_2 = T-\varepsilon$. From \cref{th:Lipschitz stability} there exists a constant $C_3>0$ such that
\[||u_0||_{L^1(\R^n)}\leq C_3||u||_{L^2(2B\times(\tau+\varepsilon,T-\varepsilon))}\leq C_3||u||_{L^2(\R^n\times(\tau+\varepsilon,T-\varepsilon))}.\]

From \cref{th:Energy} we know that there exists a constant $C=C(\varepsilon)$ such that
\[||u||_{L^2(\R^n\times (\tau+\varepsilon,T-\varepsilon))}\leq C||u||_{L^2(\omega\times(\tau,T))}\]
which concludes the proof.
\end{proof}

\begin{remark}\label{remark:C_epsilon}
The constant $C=C(\varepsilon)$ in the above inequality comes from \cref{th:Energy} and is equal to (see \cref{item:Norm L2 wrt measurements} in the proof of \cref{th:Energy})

\[C(\varepsilon) = \text{\normalfont exp}\left(\dfrac{\hat{s}K}{\varepsilon(T-\tau-\varepsilon)}\right)\dfrac{C}{\varepsilon(T-\tau-\varepsilon)}.\]

For instance, we can take $\varepsilon = (T-\tau)/4$ obtaining a constant for \cref{th:Lipschitz stability unbounded domain} of the form

\[C_2 = \text{\normalfont exp}\left(\dfrac{\hat{s}K}{(T-\tau)^2}\right)\dfrac{C_3}{(T-\tau)^2}.\] 
\end{remark}

\section{Reconstruction of the initial conditions from measurements made on a curve}\label{sec:Curve Stability}

Another problem we are interested in is a stability result for the reconstruction of the compactly supported initial temperature $u_0$ for the heat equation \cref{eq:Heat Equation} with $n=1$, from observations made on a curve contained in $\R\times [0,\infty)$ and satisfying certain properties, a problem that arises naturally from the LSFM model that shall be explained in section \cref{sec:LSFM-stability}. In this section we shall prove \cref{th:Curve stability}.

The curve where observations are available is constructed as the graph of a positive function $\sigma:\R\to\R_+$ satisfying the $\sigma-$\emph{properties} that we recall (see \cref{fig:Curve stability} as a reference):

\begin{enumerate}[i)]
\item\label{item:Regularity of sigma} $\sigma \in C^1(\R)$,
\item\label{item:Support of sigma} $\sigma>0$ for $y\in (a_1,a_2)$ and $\sigma(y)\equiv 0 $ for $y\in (a_1,a_2)^c$, for some $a_1<a_2$,
\item\label{item:Increasing and decreasing} there exists $\xi_1,\xi_2>0$ such that $\sigma'>0$ in $(a_1,a_1+\xi_1]$, $\sigma'<0$ in $[a_2-\xi_2,a_2)$ and $\sigma(a_1+\xi_1)=\sigma(a_2-\xi_2)$,
\item\label{item:Order sigma} $\dfrac{1}{\sigma'(y)} = \mathcal{O}\left(\exp\left(\dfrac{1}{\sigma(y)}\right)\right)$ as $y$ goes to $a_1^+,a_2^-$.
\end{enumerate}

\begin{figure}[h]
\centering
\includegraphics[width=\textwidth]{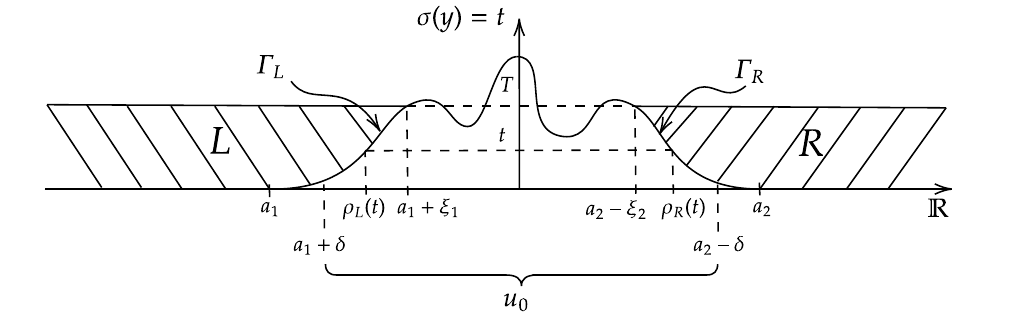}
\caption{Representation of $\sigma$-properties and the relation of $\supp(\sigma)$ with the initial condition $u_0$ for \cref{th:Curve stability}.}
\label{fig:Curve stability}
\end{figure}

Defining $T := \sigma(a_1+\xi_1)$ and as a consequence of conditions i)-iii), we can define the function $\rho_L(t):=\sigma^{-1}(t) \in C^1(0,T)\cap C[0,T]$, the inverse of $\sigma$ to the right of $a_1$, by restricting $\sigma$ to the interval $[a_1,a_1+\xi_1]$. Thus, we can parameterize the curve $\Gamma_L$ as $\{(\rho_L(t),t):0\leq t\leq T\}$. Analogously, since $\sigma$ is strictly decreasing in $[a_2-\xi_2,a_2)$, we define $\rho_R(t):=\sigma^{-1}(t) \in C^1(0,T)\cap C[0,T]$ the inverse of $\sigma$ to the left of $a_2$, then we parameterize $\Gamma_R$ as $\{(\rho_R(t),t):0\leq t\leq T\}$. The sketch of the proof is as follows: we define the set $\omega:=[a_1,a_2]^c$ as the observation region and consider the time interval of observation as $(0,T)$. From \cref{th:Lipschitz stability unbounded domain}, we are able to estimate $u_0$ with respect to the energy of $u$ in $\omega \times (0,T)$. Certainly, the energy there is less than the energy up to the curves $\Gamma_L$ and $\Gamma_R$, corresponding to the regions $L$ and $R$ in \cref{fig:Curve stability}. Consequently, to conclude \cref{th:Curve stability} we need to estimate the energy in the region $L$ with respect to the observations made on the curve $\Gamma_L$ and do the same for the region $R$. This is exactly what \cref{th:Curve estimate} establishes:

\begin{theorem}\label{th:Curve estimate}
Let $u$ be a solution of \cref{eq:Heat Equation} with $n=1$ and $u_0$ be the initial condition. Consider $\Gamma_L$ the curve constructed from the function $\sigma$ satisfying properties. If $u_0\in L^{1}(\R)$ with $\supp(u_0)\subset(a_1+\delta,a_2-\delta)$ then there exists a constant $C_9=C_9(\sigma,\delta)>0$ such that
\[\dfrac{1}{2}\dint_{0}^{T}\dint_{-\infty}^{\rho_L(\tau)}|u(y,\tau)|^2dyd\tau\leq C_9T||u_0||_{L^{1}(\R)}||u||_{L^1(\Gamma_L)}.\]
\end{theorem} 

\begin{proof}
We define the left sided exterior energy as
\[I_L(t) := \dfrac{1}{2}\dint_{-\infty}^{\rho_L(t)}|u(y,t)|^2dy,  \quad t\in [0,T),\]
and we differentiate it in order to get
\[\begin{array}{rcl} 

\dfrac{dI_L}{dt}(t)& = & \dfrac{1}{2}u^2(\rho_L(t),t)\rho_L'(t)+\dint_{-\infty}^{\rho_L(t)}u(y,t)u_t(y,t)\\

& = & \dfrac{1}{2}u^2(\rho_L(t),t)\rho_L'(t)+\dint_{-\infty}^{\rho_L(t)}u(y,t)u_{yy}(y,t)\\

&=& \dfrac{1}{2}u^2(\rho_L(t),t)\rho_L'(t)+u(\rho_L(t),t)u_y(\rho_L(t),t)-\dint_{-\infty}^{\rho_L(t)}|u_y(y,t)|^2dy.

\end{array}\]
In what follows, we shall denote $g_L(t):=u(\rho_L(t),t)$ for $t\in(0,T)$, the measurements of $u$ on $\Gamma_L$. Then
\begin{equation}\label{eq:dI/dt}
\begin{array}{rcl}
\dfrac{dI_L}{dt}(t) &=& \dfrac{1}{2} g_L^2(t) \rho_L '(t)+g_L(t)u_y(\rho_L (t),t)-\dint_{-\infty}^{\rho_L (t)}|u_y(y,t)|^2dy\\ &\leq& \dfrac{1}{2}g_L^2(t)\rho_L '(t)+g_L(t)u_y(\rho_L (t),t).
\end{array}
\end{equation}
We would like to bound the expression above so that the right-hand side depends only on the measurements $g_L$. Once we have that, we will integrate from $0$ to $t$ so that the left-hand side leads to $I_L(t)$ getting an estimate of $I_L$ in terms of $g_L$. 

For the first term in the right-hand side of \cref{eq:dI/dt} we see that \Cref{item:Support of sigma,item:Regularity of sigma} imply that $\sigma'(y)\to 0$ when $y\to a_1$ and $\rho_L'(t)\to\infty$ when $t\to 0$, thus we need to control this latter growth with the decay of $g_L(t)$ in the same limit. For the second term we directly estimate $u_y(\rho_L(t),t)$.

\begin{enumerate}[I)]
\item Let us analyze the term $g_L(t)\rho_L'(t)$ in \cref{eq:dI/dt} for $t$ in $(0,T)$, which turns out to be equivalent to study $\dfrac{g_L(\sigma(y))}{\sigma'(y)}$ for $y$ in $(a_1,a_1+\xi_1]$. Owing to the support of $u_0$ we have that
\[\begin{array}{rcl}
\left|\dfrac{g_L(\sigma(y))}{\sigma'(y)}\right| & \leq &   \displaystyle\int_{a_1+\delta}^{a_2-\delta}\frac{|u_0(r)|}{(4\pi \sigma(y))^{1/2}\sigma'(y)}\text{exp}\left(-\frac{|y-r|^2}{4\sigma(y)}\right)dr,
\end{array}\]
for $y\in(a_1,a_1+\xi_1]$. The term multiplying $|u_0(r)|$ inside the previous integral may be uniformly bounded for $(y,r)\in [a_1,a_1+\xi_1]\times [a_1+\delta,a_2-\delta]$. In effect, a singularity may occur when $y$ approaches $a_1$, but if $|a_1-y|< \delta/2$, and since $|a_1-r|\geq \delta$, then we have 
\[|a_1-r|\leq |y-r|+|a_1-y|<|y-r|+\delta/2<|y-r|+|a_1-r|/2,\]
and then
\[|y-r|>1/2|a_1-r|>\delta/2,\] 
hence
\[\dfrac{1}{\sigma(y)^{1/2}\sigma'(y)}\text{exp}\left(-\dfrac{|y-r|^2}{4\sigma(y)}\right) \leq \dfrac{1}{\sigma(y)^{1/2}\sigma'(y)}\text{exp}\left(-\dfrac{\delta^2}{\sigma(y)}\right)\]
Plugging \cref{item:Order sigma} to the previous estimate we conclude the existence of a constant $C>0$ such that
\[\left|\dfrac{g_L(y)}{\sigma'(y)}\right| \leq  C \displaystyle\int_{a_1+\delta}^{a_2-\delta}|u_0(r)| dr = C||u_0||_{L^1(\R)}.\]

\begin{comment}The following function
\[(y,r)\in (a_1,a_1+\xi_1]\times[a_1+\delta,a_2-\delta] \to \dfrac{1}{\sigma(y)^{1/2}\sigma'(y)}\text{exp}\left(-\dfrac{|y-r|^2}{4\sigma(y)}\right)\]
may be extended to $(y,r)\in [a_1,a_1+\xi_1]\times[a_1+\delta,a_2-\delta]$ by continuity thanks to property of \cref{item:Limit sigma}. In effect, a singularity may occur when $y$ approaches $a_1$, but if $|a_1-y|< \delta/2$, and since $|a_1-r|\geq \delta$, then we have 
\[|a_1-r|\leq |y-r|+|a_1-y|<|y-r|+\delta/2<|y-r|+|a_1-r|/2,\]
and then
\[|y-r|>1/2|a_1-r|>\delta/2,\] 
hence
\[\dfrac{1}{\sigma(y)^{1/2}\sigma'(y)}\text{exp}\left(-\dfrac{|y-r|^2}{4\sigma(y)}\right) \leq \dfrac{1}{\sigma(y)^{1/2}\sigma'(y)}\text{exp}\left(-\dfrac{\delta^2}{\sigma(y)}\right)\]

Consequently, this function reaches its maximum in $[a_1,a_1+\xi_1]\times[a_1+\delta,a_2-\delta]$, so
\[\left|\dfrac{g_L(y)}{\sigma'(y)}\right| \leq  C \displaystyle\int_{a_1+\delta}^{a_2-\delta}|u_0(r)| dr = C||u_0||_{L^1(\R)}.\]\end{comment}

\item Now we estimate $u_y(\rho_L(t),t)$ in $(0,T]$ for the second term in the right-hand side in \cref{eq:dI/dt}, or, equivalently, $u_y(y,\sigma(y))$ in $(a_1,a_1+\xi_1]$. First recall that
\[
u_y(y,\sigma(y)) = \dint_{a_1+\delta}^{a_2-\delta}\dfrac{u_0(r)}{\sqrt{4\pi \sigma(y)}}\text{exp}\left(-\dfrac{(y-r)^2}{4\sigma(y)}\right)\dfrac{-|y-r|}{2\sigma(y)}dr. \]
Again, the term accompanying $|u_0(r)|$ is uniformly bounded for $(y,r)\in [a_1,a_1+\xi_1]\times[a_1+\delta,a_2-\delta]$ by continuity. In conclusion,
\[|u_y(\rho_L(t),t)|\leq C\dint_{a_1+\delta}^{a_2-\delta}|u_0(r)|dr = C||u_0||_{L^1(\R)}.\]
\end{enumerate}

Bringing all the previous estimates together along with \cref{eq:dI/dt} it yields
\[\begin{array}{rcl}
\dfrac{dI_L}{dt} &\leq& \dfrac{1}{2}|g_L^2(t)||\rho_L'(t)|+|g_L(t)||u_y(\rho_L(t),t)|\\

&\leq & C||u_0||_{L^1(\R)}|g_L(t)|,
\end{array}\]
thus, integrating from $0$ to $\tau$ leads to
\[I_L(\tau) \leq C||u_0||_{L^{1}(\R)}\dint_0^{\tau}|g_L(t)|dt.\]
Integrating again in time from $0$ to $T$, we get that
\[\begin{array}{rcl}
\dfrac{1}{2}\dint_{0}^{T}\dint_{-\infty}^{\rho_L(\tau)}|u(y,\tau)|^2dy d\tau &=& \dint_{0}^{T}I_L(\tau)d\tau\\ 

\text{(Fubini)} &\leq & CT||u_0||_{L^{1}(\R)}\dint_{0}^{T}|g_L(t)| dt \\
&=& CT||u_0||_{L^{1}(\R)}||u||_{L^1(\Gamma)}.

\end{array}\]
%\[\begin{array}{rcl}
%\dfrac{1}{2}\dint_{\tau_1}^{\tau_2}\dint_{-\infty}^{\rho_L(\tau)}|u(y,\tau)|^2dy d\tau &=& \dint_{\tau_1}^{\tau_2}I_L(\tau)d\tau\\ 
%
%&\leq & C(s,\delta_s)||f_s||_{L^1(\R)}\dint_{\tau_1}^{\tau_2}\dint_{0}^{\tau}|g_L(t)|dt d\tau\\
%
%\text{(Fubini)} &=& C(s,\delta_s)||f_s||_{L^{1}(\R)}\left(\dint_{0}^{\tau_1}\dint_{\tau_1}^{\tau_2}|g_L(t)|d\tau dt + \dint_{\tau_1}^{\tau_2}\dint_{t}^{\tau_2}|g_L(t)|d\tau dt\right)\\
%
%&=& C(s,\delta_s)||f_s||_{L^{1}(\R)}\left(\dint_{0}^{\tau_1}(\tau_2-\tau_1)|g_L(t)| dt + \dint_{\tau_1}^{\tau_2}(\tau_2-t)|g_L(t)| dt\right)\\
%
%&\leq& C(s,\delta_s)  ||f_s||_{L^{1}(\R)}\left(\dint_{0}^{\tau_1}(\tau_2-\tau_1)|g_L(t)| dt + \dint_{\tau_1}^{\tau_2}(\tau_2-\tau_1)|g_L(t)| dt\right)\\
%
%&=& C(s,\delta_s)||f_s||_{L^{1}(\R)}(\tau_2-\tau_1)\dint_{0}^{\tau_2}|g_L(t)| dt
%
%\end{array}\]

\end{proof}

\begin{remark}
So far, we have estimated the energy in region $L$ (see \cref{fig:Curve stability}) with respect to the measurements available on $\Gamma_L$. Analogously, we can do the same to estimate the energy contained in region $R$ with respect to measurements available on $\Gamma_R$. Same calculations as before leads to
\[\dfrac{1}{2}\dint_{0}^{T}\dint_{\rho_R(\tau)}^{\infty}|u(y,\tau)|^2dyd\tau\leq C_9T||u_0||_{L^{1}(\R)}||u||_{L^1(\Gamma_R)}.\]

\end{remark}

We are now able to conclude the desired stability:
\begin{proof}[Proof of \cref{th:Curve stability}]
Let $\omega = [a_1,a_2]^c$. By \cref{th:Lipschitz stability unbounded domain} there exists a constant $C_2>0$ such that
\begin{equation}\label{ineq:Conclusion 1}
||u_0||_{L^1(\R)}\leq C_2 ||u||_{L^2(\omega\times(0,T))}.
\end{equation}

Moreover, \cref{th:Curve estimate} implies
\begin{equation}\label{ineq:L R Energy}
\begin{array}{rcl}
||u||^2_{L^2(\omega \times (0,T))}&\leq& \dint_{0}^{T}\dint_{-\infty}^{\rho_L(\tau)}|u(y,\tau)|^2dyd\tau + \dint_{0}^{T}\dint_{\rho_R(\tau)}^{\infty}|u(y,\tau)|^2dyd\tau\\

&\leq& C_9||u_0||_{L^1(\R)}T(||u||_{L^1(\Gamma_L)}+||u||_{L^1(\Gamma_R)}).

\end{array}
\end{equation} 
We conclude with \cref{ineq:Conclusion 1} and \cref{ineq:L R Energy}.
\begin{comment}
\begin{equation}\label{ineq:Conclusion 2}
\begin{array}{rl}
&||f_s||_{L^1(\R)}^2\leq C_7^2 ||u||^2_{L^2(\omega\times (0,T))}\leq C_7^2 C_9||f_s||_{L^1(\R)}T (||g_L||_{L^1(0,T)}+||g_R||_{L^1(0,T)}) \\ 
\Leftrightarrow & ||f_s||_{L^1(\R)} \leq C_7^2C_9T (||g_L||_{L^1(0,T)}+||g_R||_{L^1(0,T)}).
\end{array}
\end{equation}
\end{comment}
\end{proof}

\begin{remark}\label{rmk:C_vs_T_1}
The stability constant decreases with respect to $T$, which is natural from the fact that a larger $T$ means we use more information contained in our measurements. 
%the bigger is $T$, the more information from our measurements we are taken into account. 
In fact, taking $\varepsilon=T/4$ (as in \cref{remark:C_epsilon}) the constant turns out to be
\[C_9C_2^2T  = C_9\text{\normalfont exp}\left(\dfrac{\hat{s}K}{T^2}\right)\dfrac{C_3^2}{T^4}T = C_9 \text{\normalfont exp}\left(\dfrac{\hat{s}K}{T^2}\right)\dfrac{C_3^2}{T^3}. \] 

\end{remark}

\section{Stability for 2D LSFM inverse problem}\label{sec:LSFM-stability}

LSFM is an instrument that allows researchers to observe live specimens and dynamical processes by attaching fluorophores to certain cellular structures. After attaching fluorophores, the process of imaging the specimen is carried out in two steps: illumination (or excitation) and fluorescence. In the first stage a slice of the object is illuminated with a light sheet, exciting fluorophores therein. Subsequently, in the second stage, a camera measures the fluorescent radiation obtaining a two dimensional image. The microscope then scans plane by plane so that a stack of two dimensional images is collected, which represents the three dimensional object. In \cite{cueva2020mathematical} a two dimensional model is considered, hence, the illumination takes the form of a laser beam issued from different heights instead of light sheets. The Fermi-Eyges pencil-beam equation governs the illumination process, describing the space and angular distribution of photons. During the fluorescence step, photons coming out from fluorescent molecules propagate in several directions reaching the camera. The Radiative Transport Equation is used to model this second step \cite{Bal_2009}. The whole process is represented in \cref{fig:experiment}.

\begin{figure}[ht]
 \centering
 \includegraphics[width=0.7\textwidth]{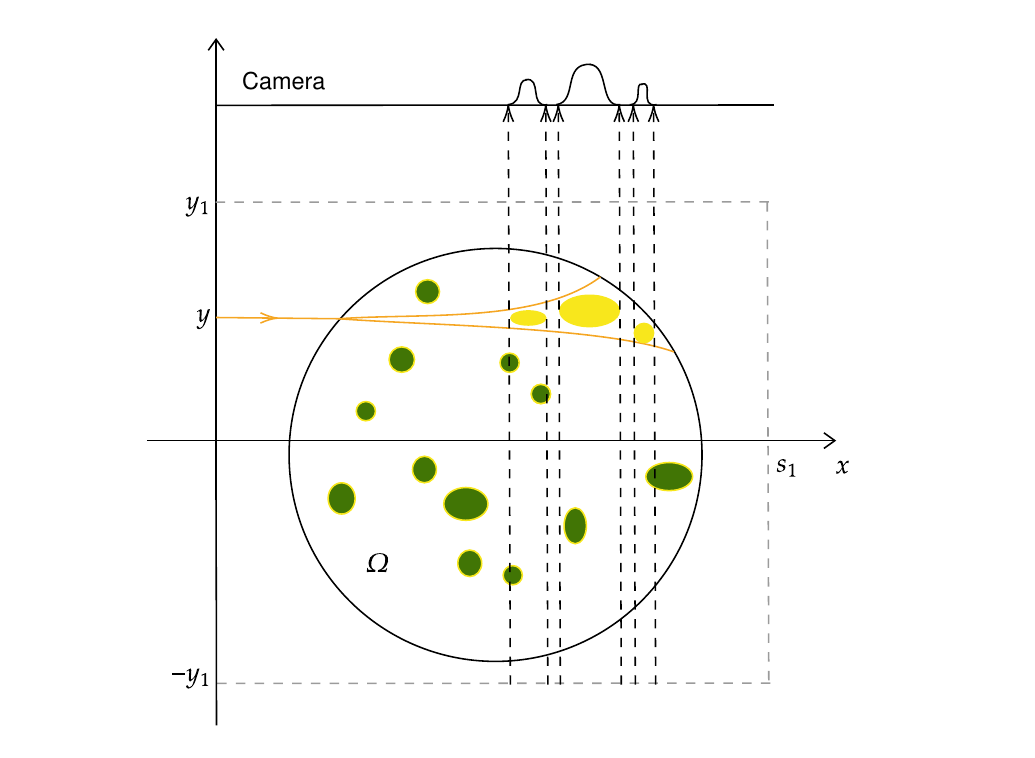}
 \caption{Representation of illumination and fluorescence stages in LSFM. A laser beam is emitted at height $y$ and illuminates the object from left. Due to scattering, photons are deflected from their original direction. Some fluorophores got excited (in yellow), the others (in dark green) will not fluoresce. Since we assume the camera is collimated, it will measure only photons emitted in straight vertical direction.} \label{fig:experiment}
\end{figure}

Let us recall some of the definitions given in \cite{cueva2020mathematical} for the setting of the LSFM model: we consider the domain $\Omega \subseteq [0,s_1]\times[-y_1,y_1]$ as the object to be observed. For $y\in[-y_1,y_1]$ we define $x_y=\text{inf}\{x:(x,y)\in\Omega\}$. For $s\in [0,s_1]$ we define
\[Y_s=\{y\in [-y_1,y_1]: x_y \leq s\},\quad
s^-= \text{inf}\{s: Y_s\neq \emptyset\}\]
% \[\begin{array}{rcl}
% Y_s&=&\{y\in [-y_1,y_1]: x_y \leq s\}\\

% s^-&=& \text{inf}\{s: Y_s\neq \emptyset\}.
% \end{array}\] 

Let $s^+$ be the largest $s$ such that $[x_y,s]\times\{y\}\subseteq \Omega$. For a fixed $s \in [s^-,s^+]$ we define $\ubar{y}=\ubar{y}(s)=\text{inf}(Y_s)$ and $\bar{y}=\bar{y}(s)=\text{sup}(Y_s)$, which, in what follows, we shall call them \textit{object top boundary} and \textit{object bottom boundary} respectively. For $s^+$ we denote $y^+=\bar{y}(s^+)$ and $y^-=\ubar{y}(s^+)$. Finally, we define the function $\gamma:Y_s\to [0,s^+]$ as $\gamma(y)=x_y$. See \cref{fig:geometric definitions} for these definitions.

There are two physical parameters involved during the illumination stage: the attenuation $\lambda$, corresponding to a measure of absorption of photons, and $\psi$ corresponding to a measure of scattering which explains the broadening of the laser beam shown in \cref{fig:experiment,fig:geometric definitions}. On the other hand, in the second stage the third physical parameter involved is the attenuation $a$, a measure of absorption of fluorescent radiation. We assume that $\lambda,a \in C_{pw}(\overline{\Omega})$, $\psi \in C^1(\overline{\Omega})$, and $\gamma\in C^1(Y_s)$. According to \cite{cueva2020mathematical}, the measurement obtained by the camera at pixel $s$ when illumination is made at height $y\in Y_s$ is given by the next expression:
\begin{equation}\label{eq:Measurements p}p(s,y) = c\cdot\text{exp}\left(-\dint_{\gamma(y)}^s\lambda(\tau,y)d\tau\right)\dint_{\R}\dfrac{\mu(s,r)e^{-\int_r^{\infty}a(s,\tau)d\tau}}{\sqrt{4\pi \sigma(s,y)}}\text{exp}\left(-\dfrac{(r-h)^2}{4\sigma(s,y)}\right)dr,
\end{equation}
where 
\begin{equation}\label{eq:sigma}
\sigma(s,y) = \dfrac{1}{2}\dint_{\gamma(y)}^s(s-\tau)^2\psi(\tau,y)d\tau.
\end{equation}

\begin{figure}[h]
\centering
\includegraphics[width=0.73\textwidth]{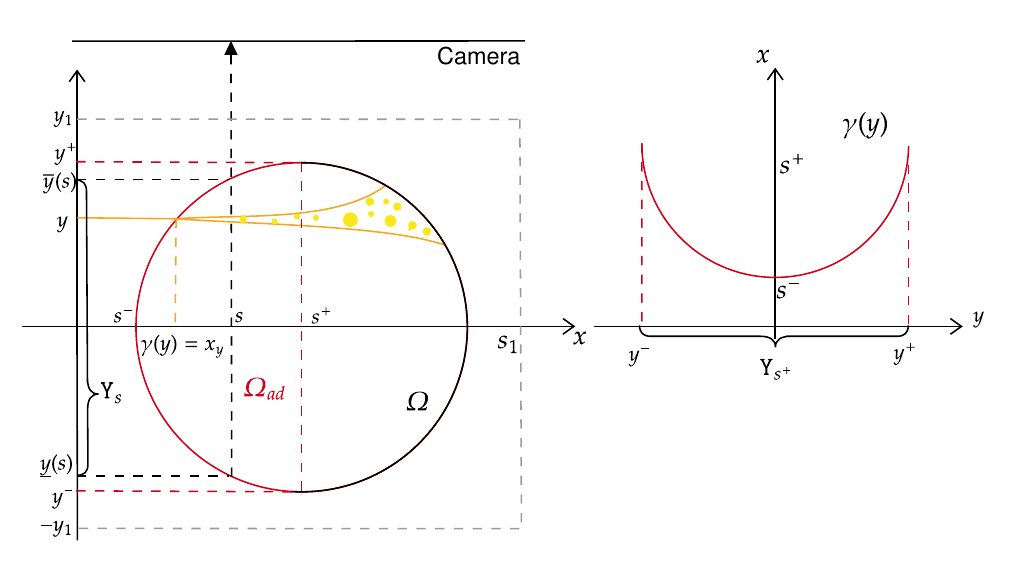}
\caption{Left figure presents the definition of the quantities $s^-$ and $s^+$ and the set $Y_{s^+}$ for a generic set $\Omega$. Right figure shows the function $\gamma$ and its domain $Y_{s^+}$ in the new coordinates.}
\label{fig:geometric definitions}
\end{figure}

In what follows, we shall fix $s$ and consider the functions $p(y) := p(s,y)$ and $\sigma(y):=\sigma(s,y)$, so that $p$ represents the measurements obtained at a pixel $s$, as a function of the height of illumination $y$. Besides, we identify $p$ and $\sigma$ with their zero-extension to the whole real line. If we consider $u$ as the solution of equation \cref{eq:Heat Equation} with $n=1$ and initial condition $u_0(y)= \mu(s,y)e^{-\int_y^{\infty}a(s,\tau)d\tau}$ then we have the following relation:
\begin{equation}\label{eq:Curve measurements}u(y,\sigma(y))=\dfrac{1}{c}\text{exp}\left(\dint_{\gamma(y)}^s\lambda(\tau,y)d\tau\right)p(y)
,\quad \forall y\in \R.\end{equation}

The above equation tells us that we have measurements of the solution of the heat equation in $\R$ on the curve  $\Gamma := \{(y,\sigma(y)):y\in \R\}\subseteq \R \times [0,\infty)$. Then, if we want a stability result for this inverse problem, it only remains to verify the hypothesis of \cref{th:Curve stability}. For this purpose, let us define a set of admissible sources: let $\widetilde{\Omega}\subsetneq\Omega$ be an open subdomain strictly contained in $\Omega$  and define $\mathcal{B}$ the set of admissible sources as (see \cref{fig:mu_delta_s}):
\begin{equation}\label{cond:suppor_mu}
\mathcal{B}:=\{\mu\in L^{1}(\R^2): \mu(s,\cdot)\in L^1(\R), \forall s\in(s^-,s^+), \supp(\mu) \subset\widetilde{\Omega}.\}
\end{equation}

The main result of this section is the following theorem

\begin{theorem}\label{th:LSFM stability}
Let $\mu\in \mathcal{B}$, $s\in(s^-,s^+)$. Then, there exists a constant $C_{10}=C_{10}(\sigma,s)>0$ such that

\[\left|\left|\mu(s,\cdot)e^{-\int_{\cdot}^{\infty}a(s,\tau)d\tau}\right|\right|_{L^1(\R)}\leq C_{10} \left(\left|\left|\frac{1}{c}p(\cdot)e^{\int_{\gamma(\cdot)}^s\lambda(\tau,\cdot)d\tau}\right|\right|_{L^1((\ubar{y},\ubar{y}+\xi_1)\cup(\bar{y}-\xi_2,\bar{y}))}\right),\]
and therefore
\[||\mu(s,\cdot)||_{L^1(\R)}\leq C_{11}||p||_{L^1((\ubar{y},\ubar{y}+\xi_1)\cup(\bar{y}-\xi_2,\bar{y}))},\]
where
\[C_{11}=\dfrac{C_{10}}{c}\exp(||a(s,\cdot)||_{L^1(\R)}+||\lambda||_{L^\infty(\R^2)}(s-s^-)).\]
\end{theorem} 

\begin{comment}
\begin{remark}
Notice that
\[e^{-\int_{y}^{\infty}a(s,\tau)d\tau}\geq e^{-||a(s,\cdot)||_{L^1(\R)}}\]
and
\[e^{\int_{\gamma(y)}^s\lambda(\tau,y)d\tau}\leq e^{||\lambda||_{L^\infty(\R^2)}(s-\gamma(y))}\leq e^{||\lambda||_{L^\infty(\R^2)}(s-s^-)},\]
hence
\[||\mu(s,\cdot)||_{L^1(\R)}\leq C_{11}||p||_{L^1((\ubar{y},\ubar{y}+\xi_1)\cup(\bar{y}-\xi_2,\bar{y}))},\]
where
\[C_{11}=\dfrac{C_{10}\exp(||a(s,\cdot)||_{L^1(\R)}+||\lambda||_{L^\infty(\R^2)}(s-s^-))}{c}.\]
\end{remark}
\end{comment}

\begin{figure}[h]
\centering
\includegraphics[width=0.71\textwidth]{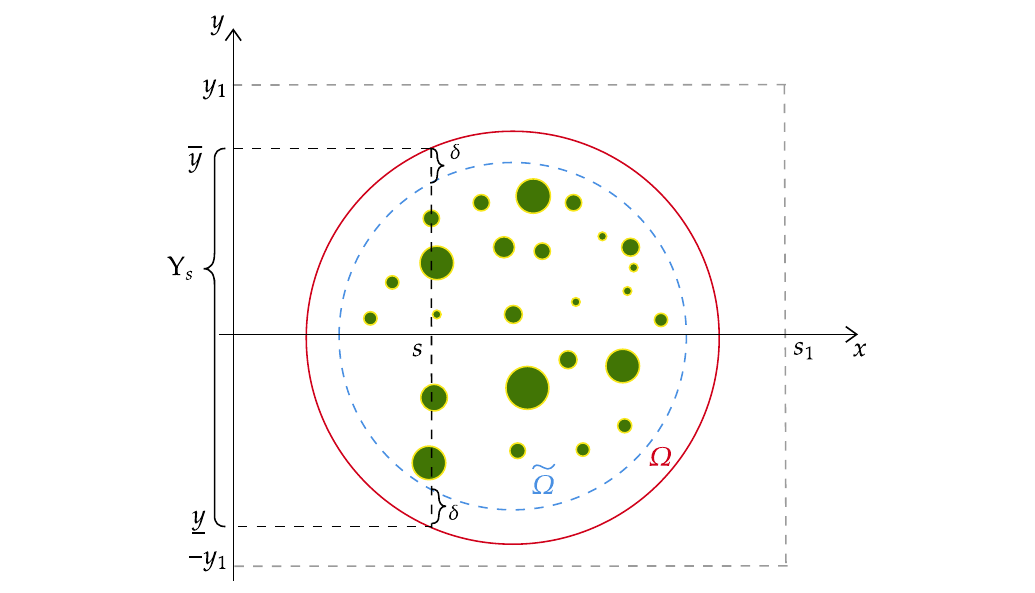}
\caption{Assumptions for \cref{th:LSFM stability}. The $\supp(\mu)$ must be far from $\partial \Omega$, which is accomplished by letting $\mu \in \mathcal{B}$.}
\label{fig:mu_delta_s}
\end{figure}

\begin{proof}
Recall that we consider $\mu(s,\cdot)e^{-\int_{\cdot}^{\infty}a(s,\tau)d\tau}$ as the initial condition of the heat equation in $\R$ and measurements are given according to \cref{eq:Curve measurements}. As in \cref{fig:mu_delta_s}, since $\mu\in\mathcal{B}$, for the fixed $s$ there exists a constant $\delta=\delta(s)>0$ such that $\mu(s,\cdot)\equiv 0$ in $(\ubar{y},\ubar{y}+\delta)\cup(\bar{y}-\delta,\bar{y})$, \emph{i.e.}, $\supp(\mu(s,\cdot)e^{-\int_{\cdot}^{\infty}a(s,\tau)d\tau})\subset (\ubar{y}+\delta,\bar{y}-\delta)$. Now, it suffices to prove that $\sigma$ satisfies the $\sigma-$properties:
\begin{enumerate}[i)]
\item From \cref{eq:sigma} we have that 
\begin{equation}\label{eq:Derivative sigma}
\sigma'(y)=-\dfrac{1}{2}\gamma'(y)(s-\gamma(y))^2\psi(\gamma(y),y)+\dfrac{1}{2}\dint_{\gamma(y)}^s(s-\tau)^2\dfrac{\partial \psi}{\partial y}(\tau,y)d\tau, \quad \text{for } y\in Y_s.\end{equation}
By the regularity of $\gamma$ and $\psi$, we get that $\sigma \in C^1(Y_s)$. Furthermore $\sigma'(\ubar{y})=\sigma'(\bar{y})=0$ since $\gamma(\bar{y})=\gamma(\ubar{y})=s$. We conclude that $\sigma \in C^1(\R)$ by noticing that $\sigma'(y)=0$ for $y\notin Y_s$.
\item From \cref{eq:sigma} and the zero-extension of $\sigma$, it is direct that $\sigma>0$ for $y\in (\ubar{y},\bar{y})$ and $\sigma(y)= 0 $ for $y\in (\ubar{y},\bar{y})^c$.
\item From \cref{eq:Derivative sigma} we get that \begin{equation}\label{eq:Bound derivative sigma}
\sigma'(y )\geq \dfrac{1}{2}(s-\gamma(y))^2\left[-\gamma'(y)\psi(\gamma(y),y)-\displaystyle\int_{\gamma(y)}^s|\psi_y(\tau,y)|d\tau\right].
\end{equation}
Let $m:=\displaystyle\inf_{(x,y)\in\overline{\Omega}}|\psi(x,y)|$ and $M:= \displaystyle\sup_{(x,y)\in\overline{\Omega}}\left|\dfrac{\partial\psi}{\partial y}(x,y)\right|$. Then
\[\dfrac{2\sigma'(y)}{(s-\gamma(y))^2}\geq -\gamma'(y)m-\dfrac{1}{3}(s-\gamma(y))M \xrightarrow[y\to\ubar{y}^+]{} -\gamma'(\ubar{y})m\]
Recalling that $\gamma'(\ubar{y})<0$ we conclude the existence of $\xi_1>0$ such that $\sigma'>0$ in $(\ubar{y},\ubar{y}+\xi_1]$. By letting $y\to\bar{y}^-$ instead of $\ubar{y}^+$ we obtain the existence of $\xi_2>0$ such that $\sigma'<0$ in $[\bar{y}-\xi_2,\bar{y})$. Furthermore,
we redefine $\xi_1$ and $\xi_2$ such that $\sigma(\ubar{y}+\xi_1)=\sigma(\bar{y}-\xi_2)=\displaystyle\min\{\sigma(\ubar{y}+\xi_1),\sigma(\bar{y}-\xi_2)\}$.
\item Finally, we not only prove that $\dfrac{1}{\sigma'(y)} = \mathcal{O}\left(\exp\left(\dfrac{1}{\sigma(y)}\right)\right)$ as $y$ goes to $\ubar{y}^+$ and $\bar{y}^-$ but $\displaystyle\lim_{y\to\ubar{y}^+}\dfrac{1}{\sigma'(y)}\exp\left(-\dfrac{1}{\sigma(y)}\right)=0$. For the limit as $y$ goes to $\bar{y}^-$ the argument is analogous. In effect, from \cref{eq:Bound derivative sigma} we get that 
\[\sigma'(y)\geq C(s-\gamma(y))^2, \quad \text{for } y\in (\ubar{y},\ubar{y}+\xi_1].\]

%Secondly, notice that
%\[\begin{array}{rcl}
%(s-\gamma(y))^2\sigma(y)^{1/2}&=&\dfrac{1}{\sqrt{2}}(s-\gamma(y))^2\left(\displaystyle\int_{\gamma(y)}^s(s-\tau)^2\psi(\tau,y)d
%\tau\right)^{1/2}\\

%&\geq& C (s-\gamma(y))^{7/2}.
%\end{array}\]

Secondly, notice that 
\[\sigma(y)=\dfrac{1}{2}\displaystyle\int_{\gamma(y)}^s(s-\tau)^2\psi(\tau,y)d\tau\leq C (s-\gamma(y))^3.\]
Then, since $\gamma(\ubar{y})=s$ we have that
\[\begin{array}{rcl}
\dfrac{1}{\sigma'(y)}\exp\left(-\dfrac{1}{\sigma(y)}\right) & \leq & \dfrac{1}{C(s-\gamma(y))^{2}}\text{exp}\left(-\dfrac{1}{C(s-\gamma(y))^3}\right) \\

& \to &  0, \quad \text{as } y\to{\ubar{y}^+},
\end{array}\]
\end{enumerate}
We conclude by applying \cref{th:Curve stability}.
\end{proof}

\begin{remark}\label{rmk:C_vs_T_2}
Certainly, the stability constant $C_{10}$ is equal to $C_4$ in \cref{th:Curve stability}. If we define $T_1 := \sigma(a_1+\xi_1)$ and $T_2 := \sigma(a_2+\xi_2)$, then we may consider the time $T = \displaystyle\min\{T_1,T_2\}$ as in \cref{fig:Gamma}. In the next section, we shall study the dependence of the stability constant with respect to this variable $T$.
\end{remark}

\begin{figure}[h]
\centering
\includegraphics[width=0.8\textwidth]{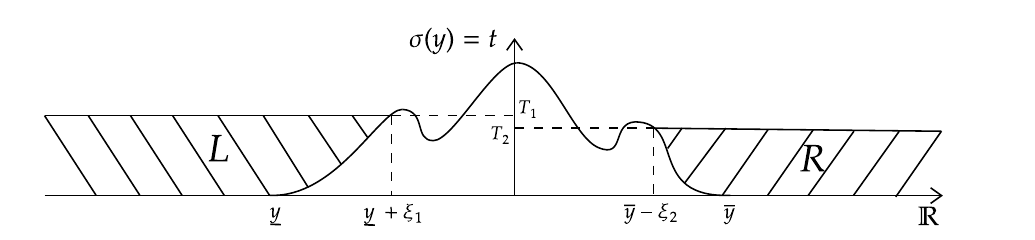}
\caption{Curve $\Gamma$ on which measurements are available for LSFM model. In this example, we must consider the variable $T=T_2$.}
\label{fig:Gamma}
\end{figure}

\FloatBarrier

\section{Numerical results in LSFM}\label{sec:Numerical results}

In this section, we analyze the behavior of the stability constant $C_{10}$ given by~\cref{th:LSFM stability}. Mainly, we observe its dependency with respect to the variable $T$, defined by $T=\min\{T_1,T_2\}$, as commented in \cref{rmk:C_vs_T_1,rmk:C_vs_T_2} (see~\cref{fig:Gamma}). We recall that the definition of $T$ depends on the monotonicity of function $\sigma$ defined in terms of the diffusion coefficient $\psi$ in~\cref{eq:sigma}. Moreover, since the result given by~\cref{th:LSFM stability} considers $\mu\in \mathcal B$, \emph{i.e.} $\supp(\mu)\subset \widetilde \Omega$, we show below that the constant $C_{10}$ increases as the support of $\mu$ gets closer to the boundary of $\widetilde\Omega$, and the stability is not guarantied when we reach  $\partial\widetilde\Omega$. We devote part of the experiments to analyze the observation interval $(\ubar{y},\ubar{y}+\xi_1)\cup(\bar{y}-\xi_2,\bar{y})$, to understand not only the stability of reconstructing $\mu(s,\cdot)$ but also, the quality of its reconstruction. 

\subsection{Datasets}
We consider three datasets as shown in~\cref{fig:datasets}. Source in \texttt{Dataset 1} describes a random distributed fluorescent sources supported in a circular domain. The attenuation $\lambda$ in the illumination stage is constant and supported in $\Omega$ with radius greater than the support of $\mu$ to guarantee the hypothesis~\cref{cond:suppor_mu}. This latter condition is also considered in the other two datasets. The source in \texttt{Dataset 2} is also randomly distributed in a support with a particular shape, this choice has the purpose of analyzing the behavior of the function $\sigma$ in terms of its increasing and decreasing intervals as we will see in \cref{ssec:top_bottom_boundaries} below. The attenuation is also constant as before. The third dataset aims to be closer to a real LSFM applications. We have simulated a \emph{zebrafish larvae} merged in an circular support with a constant attenuated substance. The source in real experiments determines, for example, zones with multicellular chemical reactions. The attenuation is composed by a constant background and a contribution given by the presence of the fluorescent source, \emph{i.e}, $\lambda = w_1\mathbbm{1}_{\widetilde \Omega} + w_2\mu$. In all cases, the diffusion term is defined by $\psi = c \lambda$, with $c>0$, which means that the diffusion is proportional to the attenuation properties of the medium. 

%To analyze the main result given by~\cref{th:LSFM stability}, on one hand, we need the initial condition $u_0=\mu(s, y)e^{-\int_y^\infty a(s,\tau)d\tau}$. We specify below how to discretize our problem depending on the analysis that we are presenting. 
\begin{figure}[ht]
%	\settoheight{\tempdima}{\includegraphics[width=.3\linewidth]{Figures/experiments/dataset1_source.eps}
%	}%
  \centering\begin{tabular}{@{}c@{ }c@{ }c@{ }c@{}}
	& Dataset 1 & Dataset 2& Dataset 3\\
\rowname{\hspace{10em}Source $\mu$}&	
 	\includegraphics[width=0.3\linewidth]{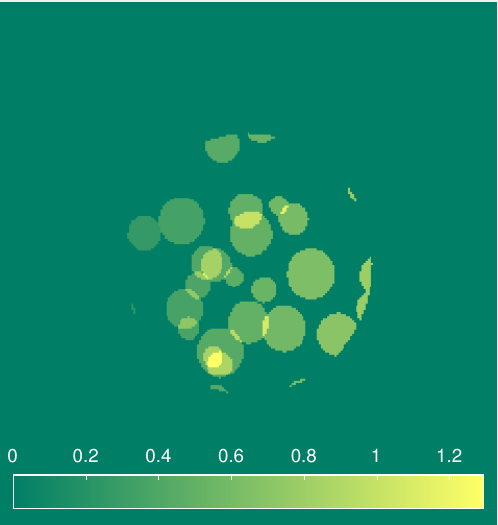}&
 	\includegraphics[width=0.3\linewidth]{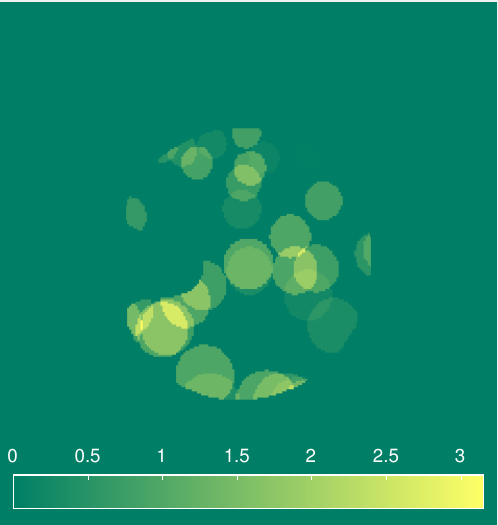}&
 	\includegraphics[width=0.318\linewidth, height=0.318\linewidth]{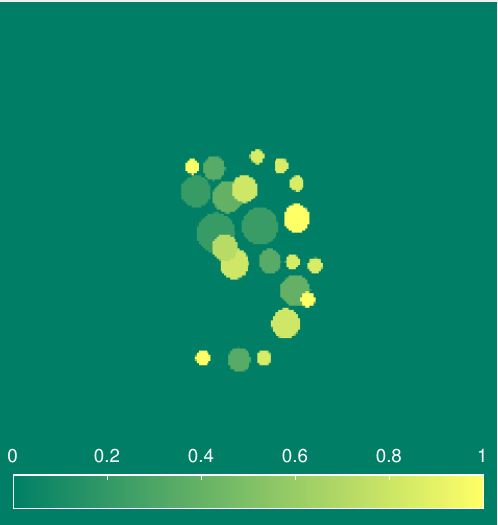}\\
\rowname{\hspace{10em}Diffusion $\psi$}&
  	\includegraphics[width=0.3\linewidth]{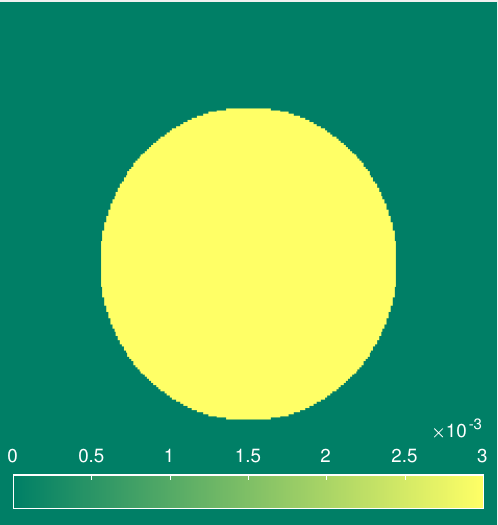}&
  	\includegraphics[width=0.3\linewidth]{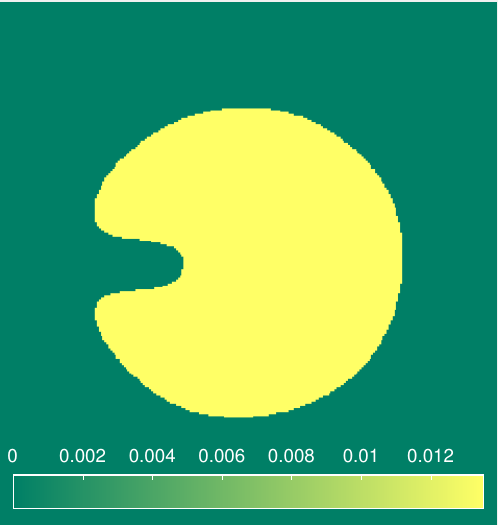}&
  	\includegraphics[width=0.318\linewidth, height=0.318\linewidth]{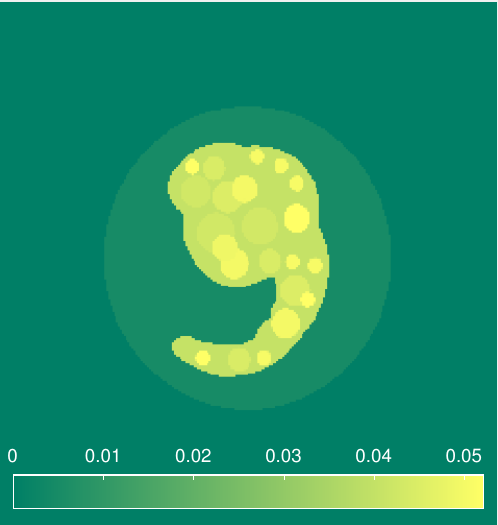}
  \end{tabular}
  \caption{Data sets. In the upper row, the sources $\mu$ for three different supports and in the bottom row, the corresponding diffusion maps. \texttt{Dataset 1} considers random distributed circles with constant attenuation. \texttt{Dataset 2} is defined on an irregular support useful to analyse the function $\sigma$. \texttt{Dataset 3} aims to be closer to a real experiment where a zebrafish embryo profile is simulated.   }\label{fig:datasets}
  \end{figure}

Our first interest is to show that the constant $C_{10}$ in~\cref{th:LSFM stability} has a relationship with the support of $\mu$, \emph{i.e.} the further we are from the boundary of $\tilde \Omega$, the better the stability of reconstructing $\mu$ is. This analysis is based on the condition number of a matrix $A_s$ that we detail below in~\cref{ssec:condition_As}. We use \texttt{Dataset 1} and \texttt{Dataset 2} to observe the proposed assay. 

\subsection{Condition number of matrix $A_s$ in terms of $\supp\mu(s,\cdot)$}\label{ssec:condition_As}
In the discrete case, as it was detailed in~\cite{cueva2020mathematical}, recovering $\mu$ is established as the solution of a linear system 
\[
\bm A \bfmu = \bb 	
\]
where $\bm A\in \mathbb R^{m\times n}$ links the vectorized source $\bm\mu\in \mathbb R^n$ to the array of measurements $\bb \in \mathbb R^m$. This is a direct consequence of the linear nature of measurements $p(s, y)$ in~\cref{eq:Measurements p} respect to the unknown variable $\mu$. 

The set of measurements considers $m_1$ heights of excitation (illuminations) and $m_2$ detectors using just one camera. The excitation process is made from right and left sides and, consequently, the number of observations is $m=2\cdot m_1\cdot m_2$. 

As we are interested on $\mu(s,\cdot)$ for a given $s\in (s^-, s^+)$ based on~\cref{th:LSFM stability}, we use the condition number of a submatrix $\bm A_s$ of $\bm A$ to know how stable is to reconstruct the restriction of $\bfmu$ to the depth $s$. This matrix $\bm A_s$ chooses the rows of $\bm A$ associated to the observations receipted by the detector $s$, one for each illumination, \emph{i.e.} $\bm A_s$ has $m_s = 2\cdot m_1$ rows. Furthermore, we want to study the stability in terms of $\supp\mu$, so we choose the columns of $\bm A$ where the support of $\mu(s,\cdot)$ is defined, this means that we focus on the pixels where the discrete source is nonzero. Observe that for larger values of the radius, more columns of $\bm A$ are taken. With this row and column sampling, we determine the submatrix $\bm A_s$ whose condition number value ($\text{cond}(A_s)$) is represented in~\cref{fig:conditionAs} for \texttt{Dataset 1} and \texttt{Dataset 2} in upper and bottom rows, respectively. 

For \texttt{Dataset 1}, the circular shape of $\supp\mu$ allows us to easily control its proximity to $\widetilde \Omega$. As it is presented in the right hand side of~\cref{fig:conditionAs}, we test radius from 0.55 until 0.8 in $\Omega = [0,2]\times [-1,1]$, the maximum value $r=0.8$ is the radius that defines $\widetilde \Omega$. As it is expected, the condition number increases when the support of $\mu$ tends to the boundary of $\widetilde \Omega$, this is shown in the left hand side of~\cref{fig:conditionAs}. We also include different values of $s$ to observe that this condition number also depends on this variable at least when the diffusion term $\psi$ is constant. The values of $s$ varies from 0.66 to 0.96, and the value of $\text{cond}(A_s)$ tends to increase when we go deeper in the object. This makes sense in the light of LSFM applications since the middle part of the object is harder to be observed directly from the measure process, and solving the inverse problem is also challenging in this zone. A similar result is observed for \texttt{Dataset 2}, we have define five different sizes of supports and five depths $s$. The condition number of the corresponding matrices $A_s$ increases when we get closer to the boundary of $\widetilde \Omega$. We also observe that the conditioning is worse for small values of $s$ compare to the previous example, this is also related to the number of illuminations in each depth $s$, we will observe this in detail in~\cref{ssec:recontruction}.
\begin{figure}[ht]
	  \includegraphics[width=0.48\linewidth]{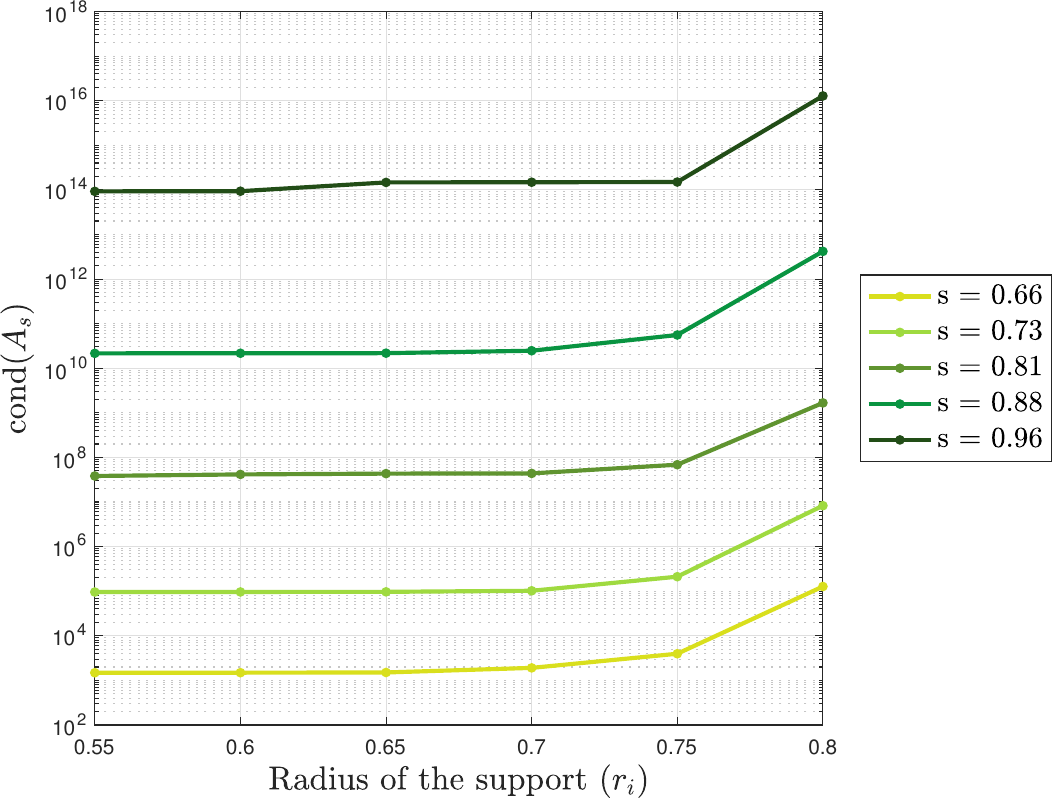}
	  \includegraphics[width=0.48\linewidth]{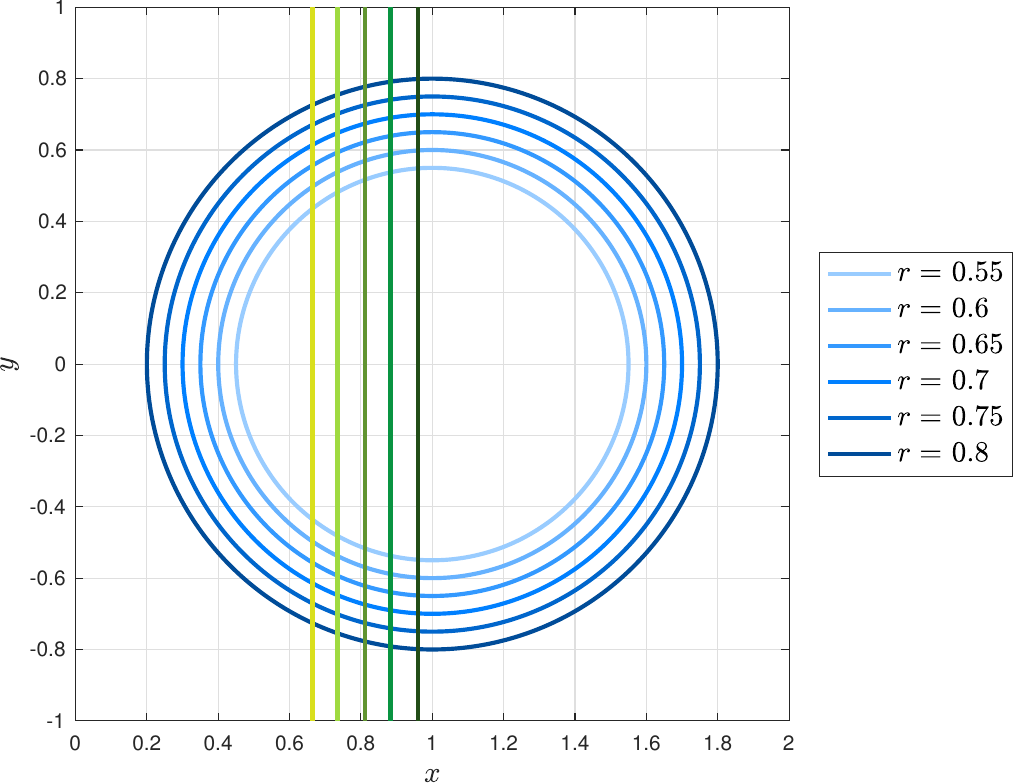}\\[15pt]
	  	  \includegraphics[width=0.48\linewidth]{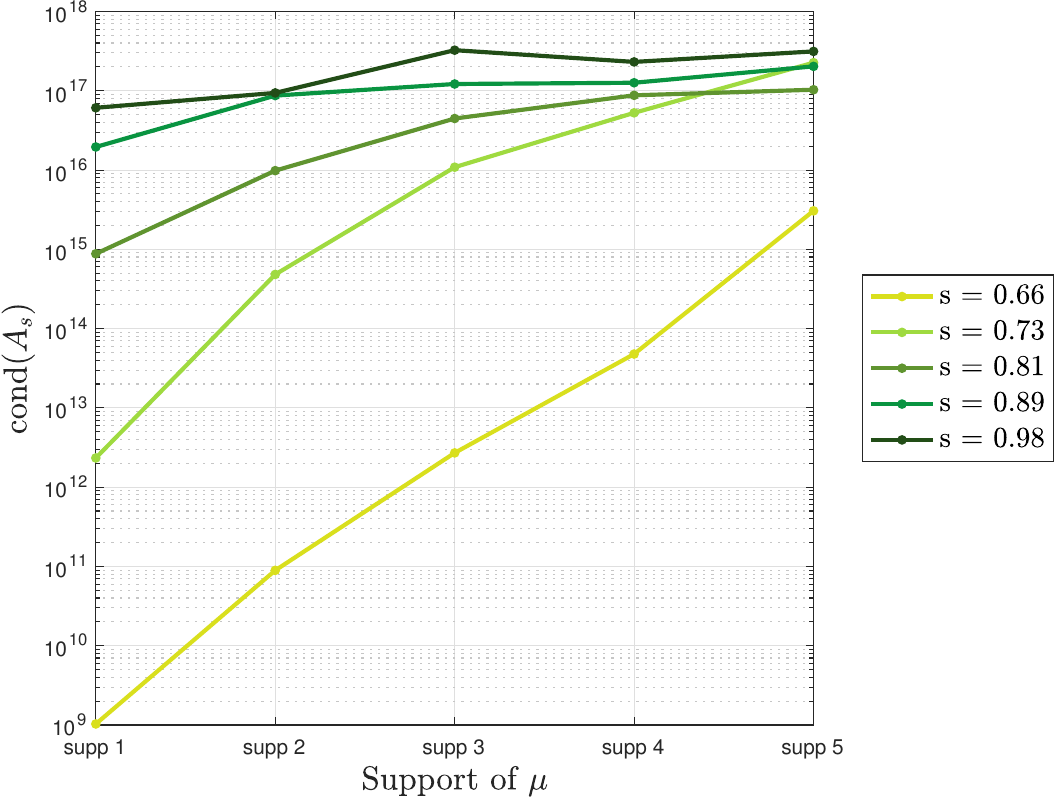}
	  \includegraphics[width=0.48\linewidth]{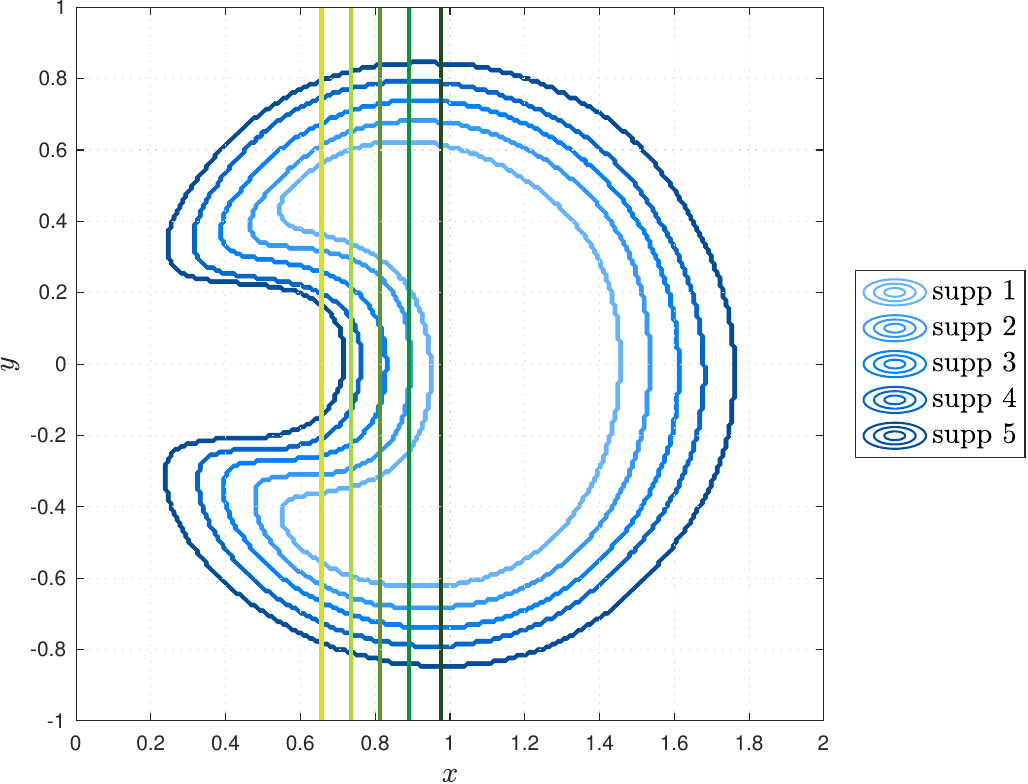}
  \caption{\emph{Left:} Condition number of matrix $A_s$ in terms of the size of the support of $\mu$ for \texttt{Dataset 1} in the upper row and, for \texttt{Dataset 2} in the bottom row. Each line is related to the depth $s$ as is shown in right figure. \emph{Right:} The different supports in terms of support and depths $s$ considered to computed the conditional number.}\label{fig:conditionAs}
\end{figure}

In the subsection below, we study in detail the observation intervals $(\ubar{y},\ubar{y}+\xi_1)\cup(\bar{y}-\xi_2,\bar{y})$ for a particular choice of $s$ using the three datasets. This will be used later to compare the condition number of $A_s$ when the illuminations are taken in the aforementioned interval or in the complete interval $(\ubar{y}, \bar{y})$.

\subsection{Object top and bottom boundaries}\label{ssec:top_bottom_boundaries}
Here, we use our three sets of data to identify the $\sigma$-properties in each case. We aim to do a representation as the one shown in~\cref{fig:Gamma}.

For \texttt{Dataset 1}, the shape of $\Gamma$ is a symmetric curve respect to the origin as is shown in~\cref{fig:gamma_dataset1}. This is a direct consequence of the constant diffusion $\psi$ and a circular domain $\widetilde \Omega$ centered in the origin. We observe that $\sigma'(y)>0$ in the interval $(-0.789,0)$ and $\sigma'(y)<0$ in $(0,0.789)$, so $T_1=T_2=3.109\times10^{-4}$ and is reached at $y=0$. According with these values, the observation set $(\ubar{y},\ubar{y}+\xi_1)\cup(\bar{y}-\xi_2,\bar{y})$ specified in~\cref{th:LSFM stability} corresponds to the interval $(-0.789, 0.789)$. 
\begin{figure}[tbhp]
	  \includegraphics[width=0.9\linewidth]{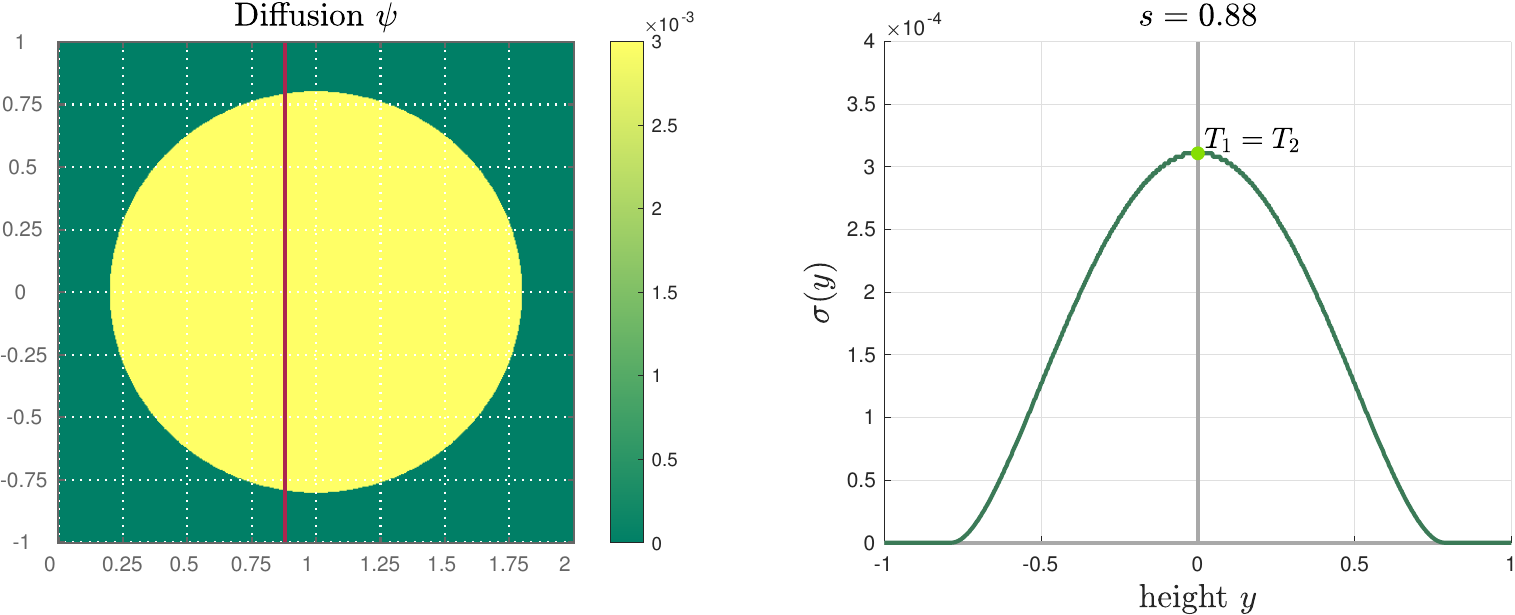}
  \caption{Curve $\Gamma$ for \texttt{Dataset 1} at $s=0.88$. On the left, the constant diffusion defined over a circle with centre in the origin and $r=0.8$. The vertical line defines the observed value of $s$. On the right, $\sigma(y)$ defines a symmetric curve $\Gamma$ where $T=T_1=T_2$. }\label{fig:gamma_dataset1}
\end{figure}

For \texttt{Dataset 2}, the curve $\Gamma$ presents a convexity near the origin as is shown in~\cref{fig:gamma_dataset2}. This behaviour is due to the particular shape of $\widetilde \Omega$, the diffusion map presents a lateral sag that is not perfectly symmetric respect to the origin in $y$-axis, as a consequence, the values of $T_1=\sigma(\ubar{y}+\xi_1)$ and $T_2=\sigma(\bar{y}-\xi_2)$ are slightly different. More precisely,  $T_1=8.48\times10^{-4}$, $T_2=8.31\times 10^{-4}$ and $T=T_2$. In this case, $\sigma'(y)>0$ in $(\ubar{y},\ubar{y}+\xi_1)=(-0.589, -0.232)$ and $\sigma'(y)<0$ in $(\bar{y}-\xi_2,\bar{y})=(0.174, 0.577)$.
\begin{figure}[tbhp]
	  \includegraphics[width=1\linewidth]{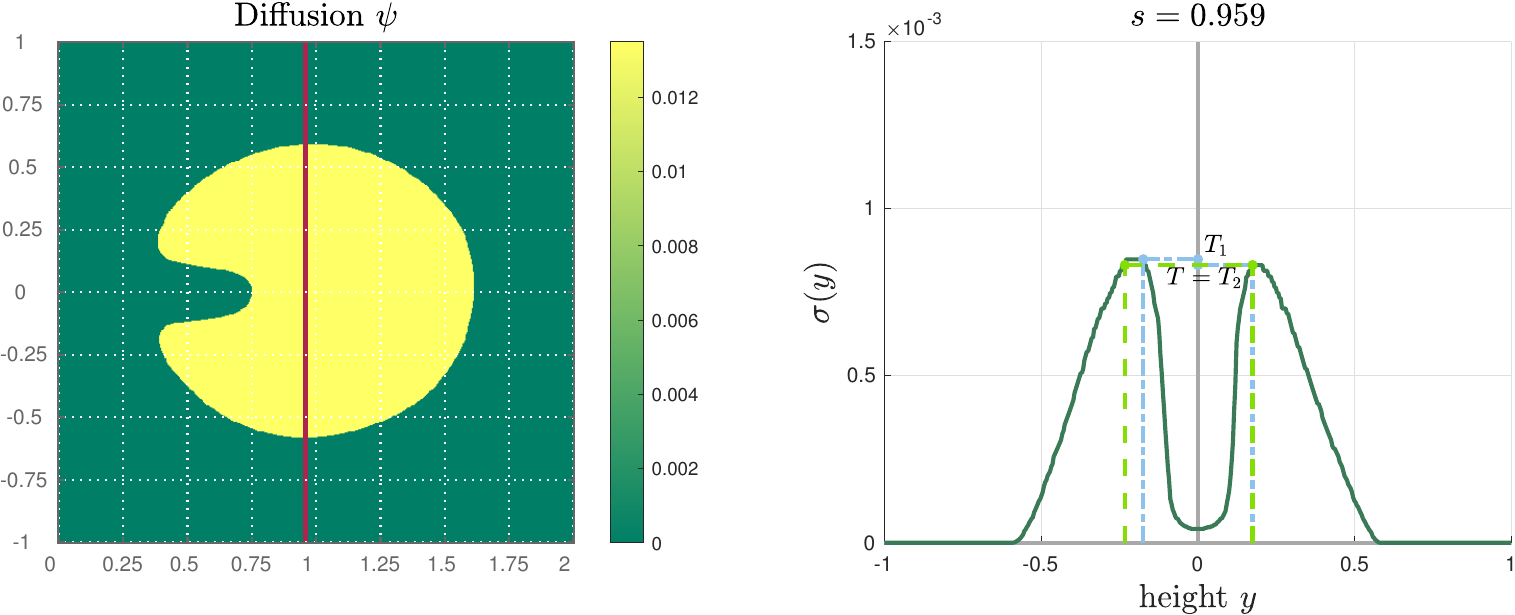}
  \caption{Curve $\Gamma$ for \texttt{Dataset 2} at $s=0.959$. On the right, the constant diffusion where the vertical line defines the observed value of $s$. On the left, $\sigma(y)$ defines the curve $\Gamma$ with a convexity around the origin. $T_1$ and $T_2$ are marked as dots and, increasing and decreasing zones are identified. }\label{fig:gamma_dataset2}
\end{figure}

For $\texttt{Dataset 3}$, the curve $\Gamma$ has a different behaviour due to the particular election of the diffusion term. As before, we have identified the illumination intervals based on the values of $T_1$ and $T_2$ as is presented in~\cref{fig:gamma_dataset3}. In this case, $T=T_2$ and $\ubar{y}+\xi_1=-0.362$ and $\bar{y}-\xi_2=0.311$. As the support of $\sigma$ is $[\underline y, \overline y]=[-0.601, 0.553]$, the illumination set in this case is defined over $(-0.601, -0.362)\cup (0.311, 0.553)$.  
\begin{figure}[tbhp]
	  \includegraphics[width=1\linewidth]{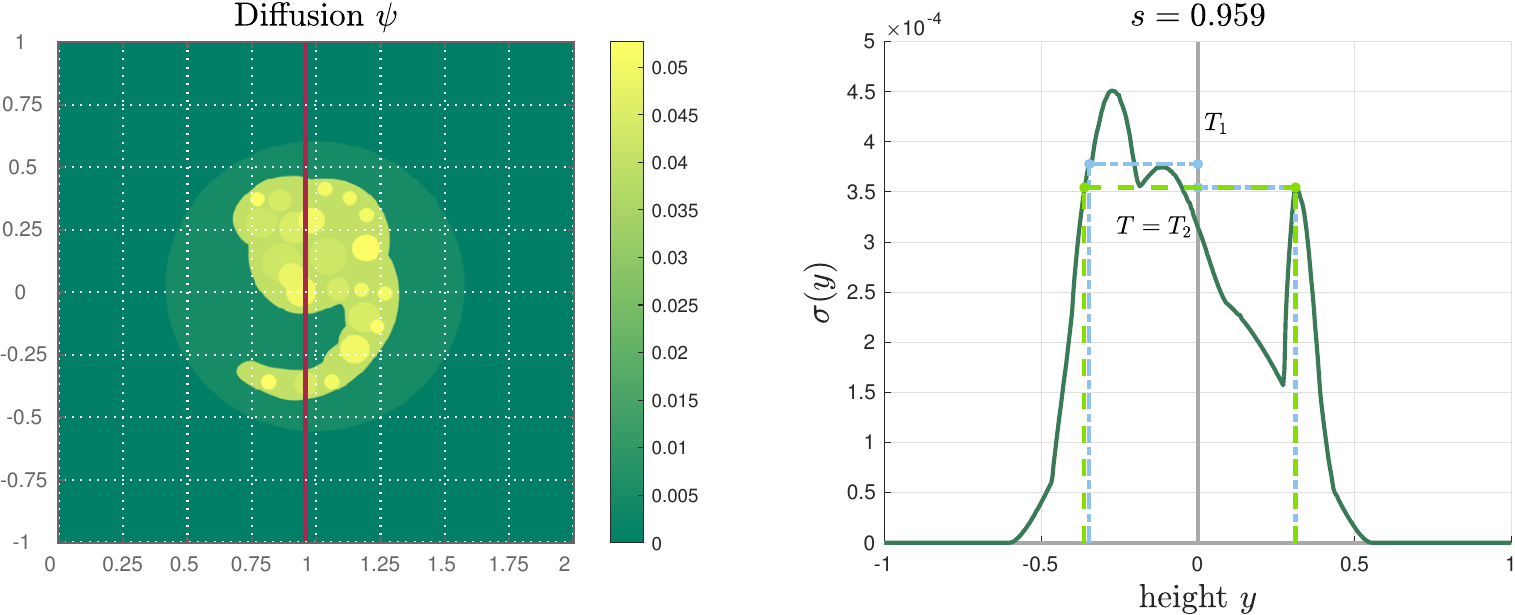}
  \caption{Curve $\Gamma$ for \texttt{Dataset 3} at $s=0.959$. On the right, the diffusion caused by the presence of the specimen, the fluorescent source and the circular medium where the zebrafish is merged. The vertical line defines the observed value of $s$. On the left, $\sigma(y)$ defines the curve $\Gamma$ and, the values of $T_1$ and $T_2$ are marked as dots. }\label{fig:gamma_dataset3}
\end{figure}

%The aim of this subsection is to numerically verify the result presented in~\cref{th:LSFM stability}. To observe the relationship between the stability constant $C_8$ and parameter $T$ defined by diffusion $\psi$, we use the \texttt{Dataset 2} with an irregular domain of observation. On the right side of~\cref{fig:psi_profile}, we observe the shape of function $\sigma$ defined by~\cref{eq:sigma} for a fixed value of $s$. Since $\sigma$ depends mainly on the function $\psi$, we included this latter on the left side in the same figure where we have highlighted the value of $s$ using a red line. 

%Based on~\cref{fig:Gamma}, we determine the values of $T_1$ and $T_2$ related to zones $L$ and $R$, where $\rho_L$ and $\rho_R$ are well defined. In this example, $T_1=4.37\times 10^{-5}$ and $T_2=7.73\times 10^{-5}$, then $T=T_2$, $\underline y+\xi_1=-0.39$ and $\overline y-\xi_2=0.39$. Then, the stability result is guaranteed only illuminating in the intervals $[-1, -0.39]$ and $[0.39, 1]$, this also determines the set of measurements $g_L$ and $g_R$ needed to estimate the constant $C_8$. 
Once we have determined the observations regions, we use this (limited-) information below to reconstruct $\mu(s,\cdot)$ and compare it with the reconstruction obtained when a full set of observations is used. This last experiment aims to show the assertion made in~\cref{rmk:C_vs_T_1}. Let us first observe the conditioning of a matrix $\bm A_s$ when full illumination are considered compared to the limited set of illuminations defined by $\sigma$-properties. For this experiment,  we have considered \texttt{Dataset 2} and \texttt{Dataset 3} where full and limited illuminations differ. In~\cref{fig:sigma_s_condition_As}, we present function $\sigma$ for different values of $s$, as in~\cref{fig:gamma_dataset2,fig:gamma_dataset3}, we determine the observation intervals that are also detailed in~\cref{tab:table_intervals}. Once, we select the illumination set, we can choose the corresponding rows of the matrix $\bm A$ to build $\bm A_s$ in each case. The condition number of $\bm A_s$ is plotted in the right hand side of~\cref{fig:sigma_s_condition_As} for \texttt{Dataset 2} in the top row and, for \texttt{Dataset 3} in the bottom row. The main difference between the dotted and continued lines is what we expected by~\cref{th:LSFM stability}, the stability of reconstructing $\mu(s,\cdot)$, observed through the condition number of $\bm A_s$, is worse when we have less observations, \emph{i.e.} when the value of $T$ is smaller as in~\cref{rmk:C_vs_T_1}. For \texttt{Dataset 3} in the full-observation case, the condition number does not have strict growth as we increase the variable $s$, this is due to the variability of the diffusion map. 
\begin{figure}[tbhp]
\centering
	  \includegraphics[width=0.45\linewidth]{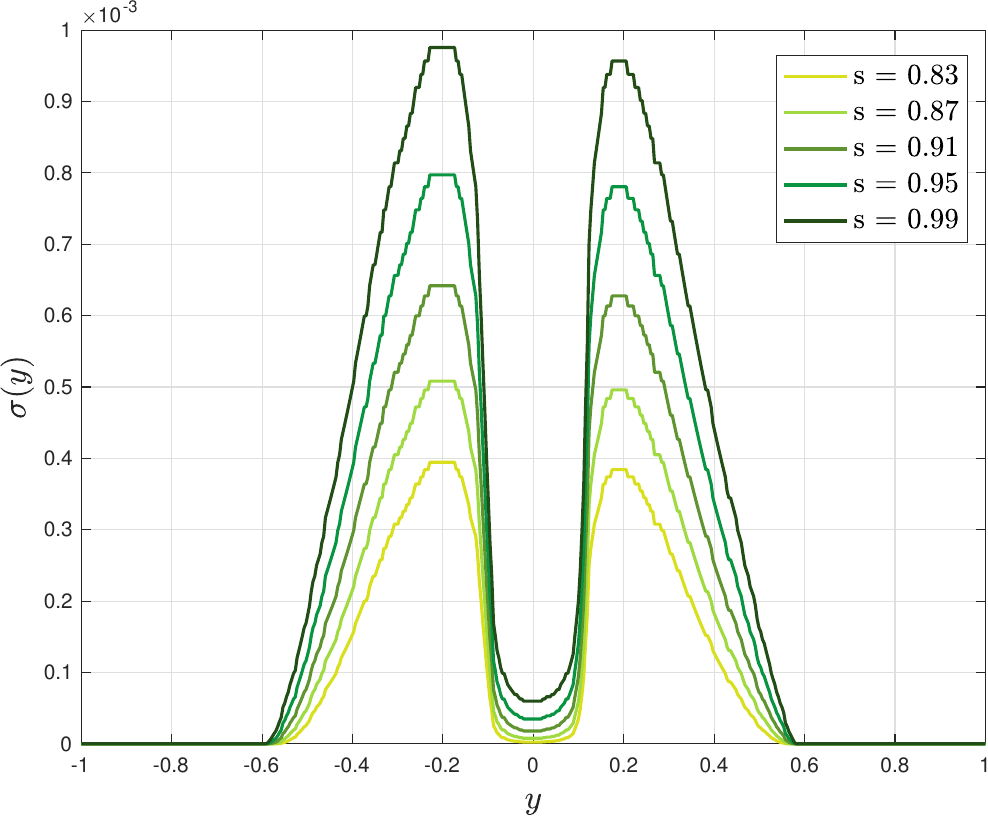}
	  \includegraphics[width=0.45\linewidth]{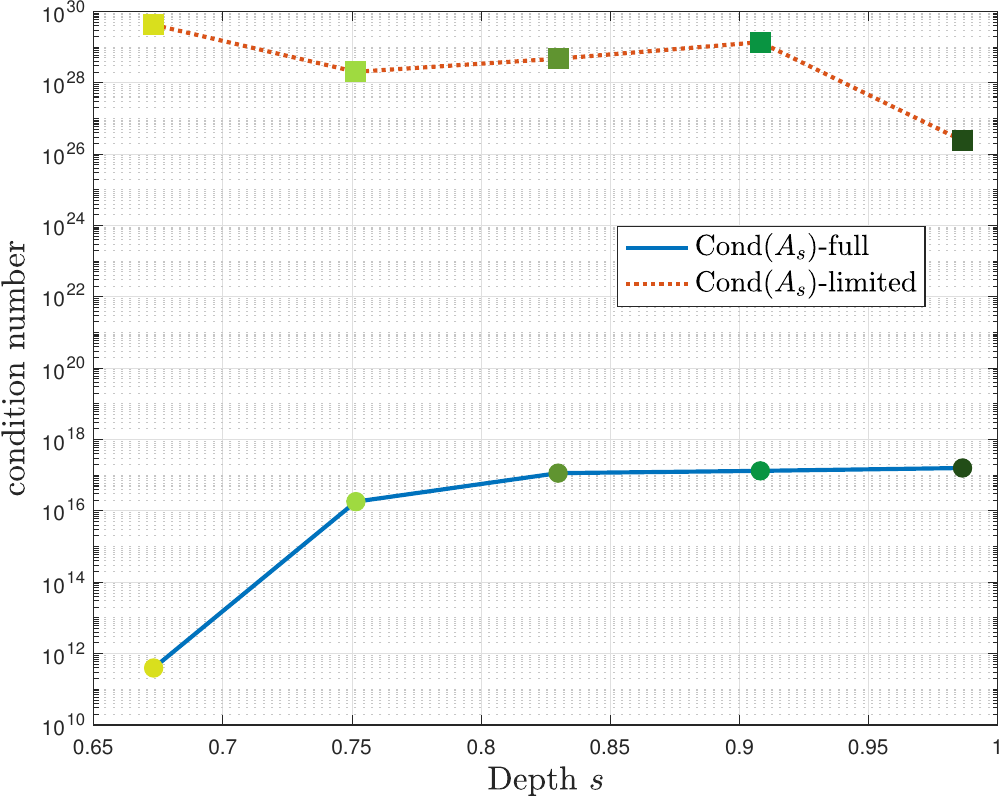}\\[15pt]
	  \includegraphics[width=0.45\linewidth]{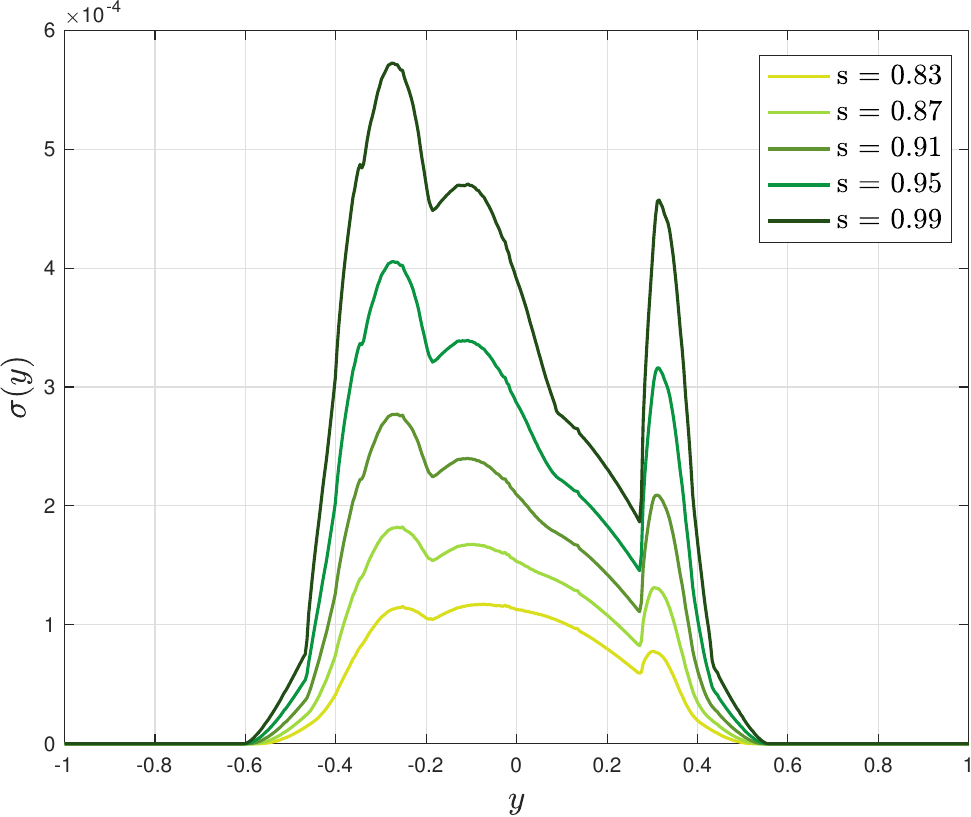}
	  \includegraphics[width=0.45\linewidth]{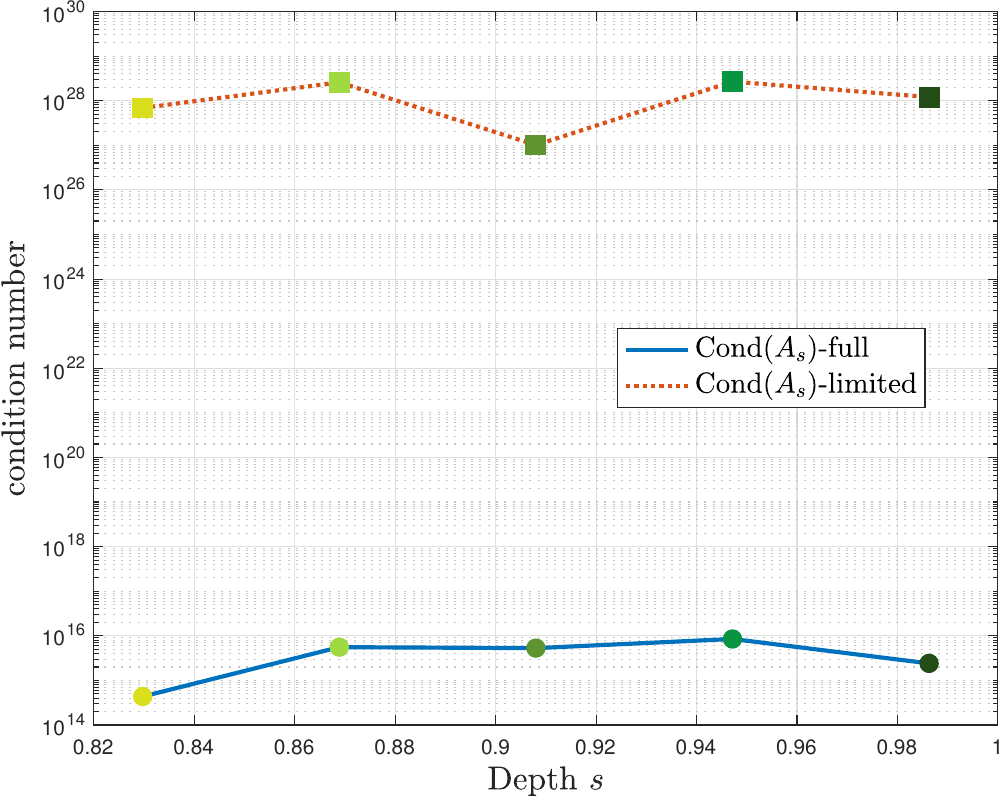}
  \caption{Condition number of $\bm A_s$ matrix. Top row: results for \texttt{Dataset 2} when different values of $s$ are considered. On the left hand, the $\sigma$ function is plotted to identify the observation intervals determined by the top and bottom boundaries. On the right hand, $\text{cond}(A_s)$ for full and limited illuminations are plotted. The bottom row considers the same results for \texttt{Dataset 3}. The limited case depends on the intervals defined in ~\cref{tab:table_intervals}.}\label{fig:sigma_s_condition_As}
\end{figure}

\begin{table}[tbhp]
{\footnotesize
  \caption{Observation intervals for \texttt{Dataset 2} and \texttt{Dataset 3}. 
  For different values of $s$ between 0.67 and 0.99, we analyze the behaviour of $\sigma$ given in~\cref{fig:sigma_s_condition_As} 
  to determine the intervals $(\ubar{y},\ubar{y}+\xi_1)\cup(\bar{y}-\xi_2,\bar{y})$ as in~\cref{fig:gamma_dataset2,fig:gamma_dataset3}.
  %$(\ubar{y},\ubar{y}+\xi_1)$
  %$\cup(\bar{y}-\xi_2,\bar{y})$ 
  %
  }\label{tab:table_intervals}
\begin{center}
\begin{tabular}{|c|c|c|} \hline
\bm{$s$}      & \bf \texttt{Dataset 2} & \texttt{Dataset 3} \\ \hline
0.67 & $(-0.566, -0.233)\cup(0.174, 0.562)$ & $(-0.577, -0.354)\cup (0.295, 0.534)$ \\[2pt]
0.75 & $(-0.577, -0.233)\cup(0.174, 0.569)$ & $(-0.597, -0.358)\cup(0.303, 0.55)$ \\[2pt]
0.83 & $(-0.585, -0.233)\cup(0.174, 0.573)$ & $(-0.597, -0.358)\cup(0.307, 0.55)$\\[2pt]
0.91 & $(-0.589, -0.233)\cup(0.174, 0.577)$ & $(-0.601, -0.362)\cup(0.311, 0.554)$\\[2pt]
0.99 & $(-0.589, -0.233)\cup(0.174, 0.577)$ & $(-0.605, -0.366)\cup(0.311, 0.558)$\\ \hline
\end{tabular}
\end{center}
}
\end{table}
Finally, we illustrate the measurement process using \texttt{Dataset 3}, we explain how to get the set of measurements after illuminating and counting photons in one camera. We select the data associated to a particular $s$ to reconstruct $\mu(s,\cdot)$ when limited and full illuminations are considered.

\subsection{Reconstructions based on parameter $T$}\label{ssec:recontruction}
In this part, we aim to reconstruct $\mu(s,\cdot)$ for a fixed value of $s$ using \texttt{Dataset 3}. We will see that this reconstruction is stable in terms of~\cref{th:LSFM stability}. The resulting set of measurements after illuminating along all possible heights is represented in~\cref{fig:dataset3_measurements}. We can observe a blurred image which represents the effects of the diffusion (scattering) during the excitation stage. These measurements were also perturbed by Poisson noise to avoid inverse crime during the reconstruction.
% In~\cref{fig:dataset3_illumination}, we present the excitation stage for an specific height of illumination, the distribution of the laser is described by Fermi--Eyges pencil-beam equation as it was detailed in~\cite{cueva2020mathematical} and introduced in~\cref{sec:LSFM-stability}. 
% \begin{figure}[tbhp] 
% \settoheight{\tempdima}{\includegraphics[width=.3\linewidth]{Figures/experiments/dataset3_source_lined.eps}}%
% \centering
% \includegraphics[width=0.24\linewidth]{Figures/experiments/dataset3_source_lined.eps}
% \includegraphics[width=0.24\linewidth]{Figures/experiments/dataset3_atenuation_lined.eps}
% \includegraphics[width=0.24\linewidth]{Figures/experiments/dataset3_vh_ill_110.eps}
% \includegraphics[width=0.24\linewidth]{Figures/experiments/dataset3_wh_ill_110.eps}
% \caption{Illumination stage at height $h=-0.1484$ for \texttt{Dataset 3}. From left to right: fluorescent source crossed by a line at illumination height. Secondly, attenuation $\lambda$, marking the value $x_h=0.4297$ as a vertical line, where the laser enters by first time to the object. The third image shows the total density of excitation photons ($v_h$) and the last image displays the amount of excited fluorophores under the relation $w_h(x, y) = c\ v_h(x, y)\mu(x, y)$, for $c>0$. }\label{fig:dataset3_illumination}
% \end{figure}

\begin{figure}[tbhp]
\settoheight{\tempdima}{\includegraphics[width=.3\linewidth]{Figures/experiments/dataset3_measurements}}%
\centering
\includegraphics[width=0.35\linewidth]{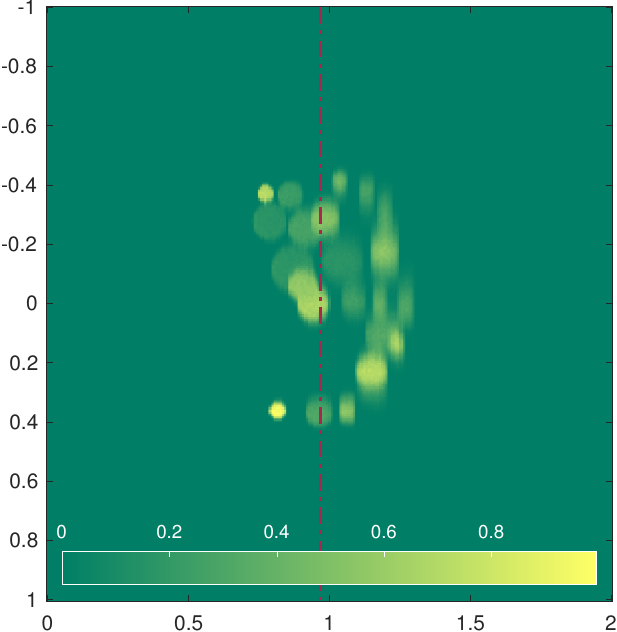}
\caption{Noise measurements for \texttt{Dataset 3}. For each illumination height $y$ in $y$-axis, the corresponding row in the image represents the number of photons (scaled as intensity) that is observed by the camera after the excitation beam is emitted at the point $(0, y)$. The vertical line marks the depth $s$ considered to reconstruct $\mu(s,\cdot)$. }\label{fig:dataset3_measurements}
\end{figure}

In~\cref{fig:dataset3_reconstructions}, we present the reconstruction of $\mu(s,y)$ for $s=0.969$, the left hand side of this figure presents the $\sigma$ function needed to determine the object top and bottom boundaries. As was analyzed before, the illuminations used in~\cref{th:LSFM stability} that determine measurements are taken in the set $I=(-0.602, -0.375)\cup(0.344, 0.555)$. On the right side of~\cref{fig:dataset3_reconstructions}, we show the source $\mu(s,y)$ as ground truth, the reconstruction using only illuminations over $I$ (limited illuminations) and, the reconstruction for illuminations over $(-0.602, 0.555)$ (full illuminations). These reconstructions are associated to the solution of a linear system of the form $\bm A_s \bm\mu_s= \bm b_s$ that were solved using \emph{simultaneous algebraic reconstruction technique method} (sart) provided by the MATLAB package IR Tools~\cite{gazzola2018}. We observe that a stable reconstruction is possible in both cases but the lack of information in the limited-illumination case produces a gross reconstruction in the non-observable region.
\begin{figure}[tbhp]
\centering
\includegraphics[width=\linewidth]{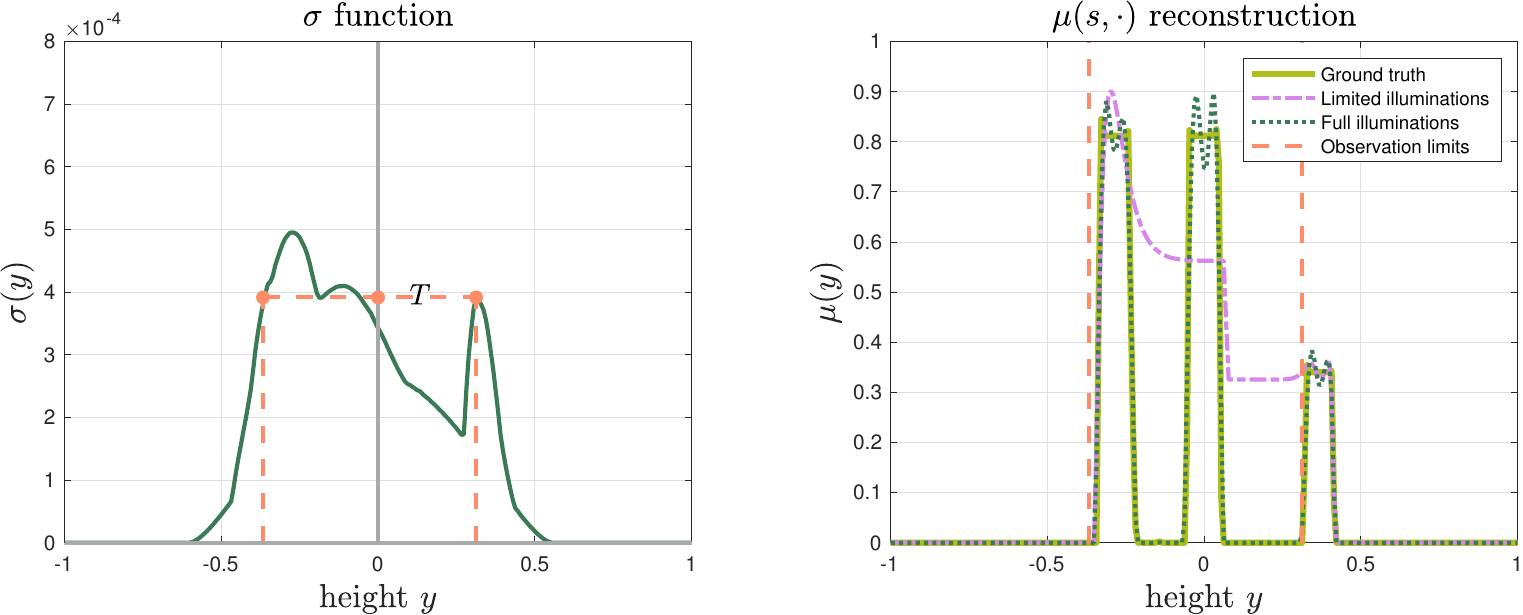}
\caption{Reconstruction of $\mu(s,\cdot)$ for $s=0.969$ using \texttt{Dataset 3}. On the left, $\sigma$ function for the selected $s$ with the corresponding observation interval whose limits are represented by a dashed line. On the right, $\mu(s,\cdot)$ and its limited and full reconstructions are added as profiles for ease of comparison.}\label{fig:dataset3_reconstructions}
\end{figure}

%\newpage

\section{Conclusions}
\label{sec:conclusions}

Two novel results has been established with respect to the stability for the reconstruction of the initial temperature for the heat equation in $\R^n$, for distinct domains of observation of the form $\omega \times (\tau,T)$. Typical results in the literature for the backward heat equation problem provide us with logarithmic estimates when incorporating some a priori information on the initial condition. In our case, we have been able to improve those estimates to a Lipschitz one, at least for compactly supported initial conditions. We expect that this result may be extended for more general initial temperatures. Furthermore, another interesting stability result is obtained for the reconstruction of the initial temperature for the heat equation in $\R$ when measurements are available on a curve $\Gamma \subset \R\times[0,\infty)$, a problem that arises from the LSFM model established in \cite{cueva2020mathematical}. However, we have to be careful with these results, more specifically, we highlight the stability constant. Recall that this constant comes, in part, from the open mapping theorem, which ensure just the existence of this term, without giving any information about the dependency on the parameters of the problem. Consequently, if this constant is too large in comparison to the noise level in the measurements, then we can not expect a good reconstruction from the numerical point of view, despite the Lipschitz estimate. In fact, as numerical results indicate, even though a small noise is added to the measurements, the initial condition reconstructed is away from the real one in those sections where measurements are not taken into account, which give insights of a high value for the stability constant. We expect that the result may be improved by considering all the measurements available and not restricting to those heights of illumination for which $\sigma$ is increasing or decreasing.

%\begin{comment}
\appendix
\section{Erratum}\label{erratum}
\subsection{Introduction}

Motivated by applications to an inverse problem arising in fluorescence microscopy, in the paper \cite{arratia2021} the authors pursue stability inequalities for the backward heat equation problem with outer measurements. It was observed that such outer-measurement estimates lead to stability inequalities for the heat equation with measurements along a suitable space-time curve. Based on the approach introduced by the authors in \cite{cueva2020mathematical} ---where the uniqueness question for the inverse problem in fluorescence microscopy is addressed--- it is clear that the later estimate allows to establish the sought stability for the inverse fluorescence microscopy problems.

The first result presented in \cite{arratia2021} is a conditional logarithmic stability inequality for the case of outer measurements, and it is proved using Carleman estimates. An importante feature of this approach is the fact that it provides rather explicit constants in the estimates.

On the other hand, inspired by previous examples of Lipschitz stability inequalities for the heat equation such as \cite{Saitoh_Yamamoto_1997}, we attempted to use a different methodology, namely, a compactness-uniqueness argument, with the purpose of obtaining a conditional Lipschitz stability in our specific outer-measurement setting. 

We are submitting this corrigendum because the mentioned argument is incorrect, invalidating theorem 1.3 and those that follow from it (theorems 1.2, 1.4 and 5.1). The error resides on the nonlinear nature of the operators that take positive or negative parts of a given function.

Despite the previous issue, one can still propose considerably weaker Lipschitz stability estimate, but with a more explicit constant, which is a feature shared with the method employed in the deduction of the conditional logarithmic inequality, although the later estimate is established for more natural Sobolev norms.

The corrigendum is organized as follows. In section 2 we state corrected versions of the theorems that had to be revised due to the error. In addition, and for the sake of completeness, we included in section 3 a counterexample to  theorem 1.3 in \cite{arratia2021}, while section 4 contains brief proofs of the revised versions of the main theorems. We finalize this corrigendum with some comments on the numerical experiments.

\subsection{Revised theorems}

We have decided to keep the enumeration of theorems as they appear in \cite{arratia2021}. The main modifications are as follow. Considering the initial value problem
\begin{equation}\label{eq:Heat Equation2}
\left\{ \begin{array}{rcll} 
u_t-\Delta u & = & 0 & \text{in } \mathbb{R}^n\times(0,T),\\
u(y,0) & = & u_0(y) & \text{in } \mathbb{R}^n,\\
\displaystyle\lim_{|y|\to\infty}u(y,t) & = & 0 & t\in (0,T),
\end{array}\right.
\end{equation}
theorem 1.3 and 1.2 in \cite{arratia2021} are replaced, respectively, by the next weaker and straightforward conditional Lipschitz stability results. Below, $H^{-2}(B)$ stands for the topological dual of $H^{2}_0(B)$, with this space corresponding to the completion of $C^\infty_c(\Omega)$-functions with respect to the $H^2$-norm. We consider instead the norm $\|\varphi\|_{H^{2}_0(B)}=\|\Delta\varphi\|_{L^2(B)}$ thanks to Poincare's inequality.

%\begin{theorem}\label{th:Lipschitz stability}
\begin{customthm}{1.3}\label{th:Lipschitz stability2}
Let $B\subset\RR^n$ be a bounded open subset. If $u_0\in L^2(\RR^n)$ with $\supp(u_0)\subseteq B$ and $u$ is the corresponding solution to \eqref{eq:Heat Equation2}, then 
%there exists a constant $C_3 = C_3(R,t_1,t_2)>0$, for $0<t_1<t_2$, such that
%\[||u_0||_{H^{-2}(B)}\leq C_1||u||_{L^1(B\times(0,T))}\]
%and
\[\|u_0\|_{H^{-2}(B)}\leq C\|u\|_{L^2(B\times(0,T))},\]
with $C=(C_PT^{-1/2}+T^{1/2})$ and $C_P>0$ a Poincare's inequality constant such that $\|\varphi\|_{L^2(B)}\leq C_P\|\Delta\varphi\|_{L^2(B)}$ for all $\varphi\in C^\infty_c(B)$.

%\end{theorem}
\end{customthm}

%\begin{theorem}\label{th:Lipschitz stability unbounded domain}
\begin{customthm}{1.2}\label{th:Lipschitz stability unbounded domain2}
For $R>0$ we set $B:=B(0,R)$, the ball of radius $R$ centered at the origin. Let $T>0$ and $\omega\subseteq\RR^n$ be such that $\RR^n\backslash\omega$ is compact and $B\subseteq \RR^n\backslash \omega$. Let $u_0\in L^2(\RR^n)$ satisfy with $\supp(u_0)\subseteq B$ and let $u$ be the respective solution of \eqref{eq:Heat Equation2}. Then there exists a constant $C=C(R,T,\omega)>0$ such that
\[\|u_0\|_{H^{-2}(B)}\leq C\|u\|_{L^2(\omega\times(0,T))}.\]
%\end{theorem}
\end{customthm}

On the other hand, Theorem 1.4 in the original paper needs to be revised as well and now takes the form of a weaker conditional Lipschitz stability and a conditional logarithmic stability estimate. We follows the definitions and notations introduced in the original paper. In order to write the results in a more concise fashion, we introduce the norm $\|\cdot\|_{m}$ and the modulus of continuity $m:\RR_+\to\RR_+$ which are defined as
\[
\|\cdot\|_{m}=\|\cdot\|_{H^{-2}(B)} \quad\text{if}\quad m(x)=x\quad\text{and}\quad \|\cdot\|_{m}=\|\cdot\|_{L^2(B)} \quad\text{if}\quad m(x)=(-\log(x))^{\kappa}.
\]

%\begin{theorem}\label{th:Curve stability}
\begin{customthm}{1.4}\label{th:Curve stability2}
Let $\sigma:\RR\to\RR_+$ be a function satisfying the  $\sigma$-properties. Let $u$ be the solution of \eqref{eq:Heat Equation} in space-dimension $n=1$, for some $u_0\in L^2(\RR)$ (resp. $u_0\in \mathcal{A}_{\beta,M}=\{a \in H^{2\beta}(\RR): \|a\|_{H^{2\beta}(\RR)}\leq M\}$) such that $\supp(u_0)\subset B=(a_1+\delta, a_2-\delta)$, where $0<\delta<(a_2-a_1)/2$. Let $\Gamma_L:=\{(y,\sigma(y)):y\in(-\infty,a_1+\xi_1]\}$ and $\Gamma_R:=\{(y,\sigma(y)):y\in[a_2-\xi_2,\infty)\}$ be two curves contained in $\RR\times[0,\infty)$ where measurements are available. Then there exist $\kappa\in(0,1)$ and a constant $C = C(\sigma,\delta)>0$ such that
\[\|u_0\|_{m}\leq C m(\|u\|_{L^1(\Gamma_L\cup\Gamma_R)}).\]
%Moreover, assuming as well that $u_0\in\mathcal{A}_{\beta,M}$ and $||u||_{L^1(\Gamma_L\cup\Gamma_R)}<1$, there exist $\kappa\in(0,1)$ and another constant $C'=C'(\sigma,\delta)>0$ such that
%\[||u_0||_{L^2(\RR)}\leq C'\left(-\log||u||_{L^1(\Gamma_L\cup\Gamma_R)}\right)^{-\kappa}.\]
%\end{theorem}
\end{customthm}
For the definition of the $\sigma$-properties, we refer the reader to the paper \cite{arratia2021}.

Finally, in relation to the inverse problem in Light Sheet Fluorescence Microscopy, we revise the conclusions of Theorem 5.1 in \cite{arratia2021} and present the following corrected version. Once again, we follow the definitions and notations presented in the original article.

%\begin{theorem}\label{th:LSFM stability}
\begin{customthm}{5.1}\label{th:LSFM stability2}
Let $\mu\in \mathcal{B}$ (resp. $\mu\in \mathcal{B}\cap\mathcal{A}_{\beta,M}$), with
\[\mathcal{B}:=\{\mu\in L^{2}(\RR^2): \mu(s,\cdot)\in L^2(\RR), \forall s\in(s^-,s^+), \supp(\mu) \subset\widetilde{\Omega},\}\]
and let $s\in(s^-,s^+)$. Then, there exist $\kappa\in(0,1)$ and a constant $C=C(\sigma,s)>0$ such that

\[\left|\left|\mu(s,\cdot)e^{-\int_{\cdot}^{\infty}a(s,\tau)d\tau}\right|\right|_{m}\leq C m\left(\left|\left|\frac{1}{c}p(\cdot)e^{\int_{\gamma(\cdot)}^s\lambda(\tau,\cdot)d\tau}\right|\right|_{L^1((\ubar{y},\ubar{y}+\xi_1)\cup(\bar{y}-\xi_2,\bar{y}))}\right),\]
and therefore
\[\|\mu(s,\cdot)\|_{m}\leq C'm\big(\|p\|_{L^1((\ubar{y},\ubar{y}+\xi_1)\cup(\bar{y}-\xi_2,\bar{y}))}\big),\]
where
\[C'=\dfrac{C}{c}\exp(\|a(s,\cdot)\|_{L^1(\RR)}+\|\lambda\|_{L^\infty(\RR^2)}(s-s^-)).\]
%\end{theorem} 
\end{customthm}

\subsection{Counterexample for Theorem 1.3}

%The result of Theorem 1.3 actually cannot hold true, due to the completeness of $L^p$ spaces, the regularization of solutions by the heat equation and Sobolev embedding estimates.

%Next we also present an explicit example showing some of the obstacles to a result like Theorem 1.3 and the kind of conditions necessary for a conditional Lipschitz estimate.
The next example illustrates some of the obstacles that prevent the validity of a Lipschitz (unconditional) inequality for the backward heat propagation problem when observations take place in an exterior region.

\begin{Example}
For $k\in\mathbb{N}$ let $u_k(y,t)$ be the solution of $(P_k)$ the heat equation in $\mathbb{R}\times(0,\infty)$ described as
\begin{align*}
(P_k)\qquad \left\{ \begin{array}{rcll} 
u_t-\Delta u & = & 0 & \text{in } \mathbb{R}\times(0,\infty),\\
u(y,0) & = & g_k(y) & \text{for } y\in\mathbb{R},\\
\displaystyle\lim_{|y|\to\infty}u(y,t) & = & 0 & t\in (0,\infty).   
\end{array}\right.
\end{align*}
for $g_k(y)$ defined as
\begin{align*}
\begin{cases}
g_k(y)=\sin(ky), & \textnormal{ for } |y|\leq \pi,\\
g_k(y)=0, & \textnormal{ for } |y| > \pi.
\end{cases}.
\end{align*}
Then, $\|g_k\|_{L^1(\RR)}=4, \forall k\in \mathbb{N}$, while $\|u(\cdot,t)\|^2_{L^2(\RR)}\leq \tilde{C}\left(e^{-k^2t}+\frac{1}{k}\right)$, for all $t>0$, and $\tilde{C}$ being an absolute constant independent of $k$ and $t$.
\end{Example}

Let us verify these claims. On one hand, we can immediately verify that
\begin{align*}
\|g_k\|_{L^1(\mathbb{R})}=\int_{-\pi}^\pi |\sin(ky)|dy = \int_{-k\pi}^{k\pi} |\sin(x)|\frac{1}{k} dx = 2\int_0^\pi \sin(x) dx =4.
\end{align*}

On the other hand, the solution $u_k(y,t)$ can be written as
\begin{align*}
u_k(y,t) &= \frac{1}{\sqrt{4\pi t}}\int_\RR g_k (x) \exp\left({-\frac{(y-x)^2}{4t}}\right)dx \\
&= \frac{1}{\sqrt{4\pi t}}\int_{\pi}^\pi \sin(kx) \exp\left({-\frac{(y-x)^2}{4t}}\right)dx
\end{align*}
and we will analyze the solution $u_k(y,t)$, for $t>0$ fixed, for $y$ in different intervals of $\RR$, as indicated in Figure \ref{fig:regions}
\begin{figure}[h!]
\begin{center}\includegraphics[scale=.4]{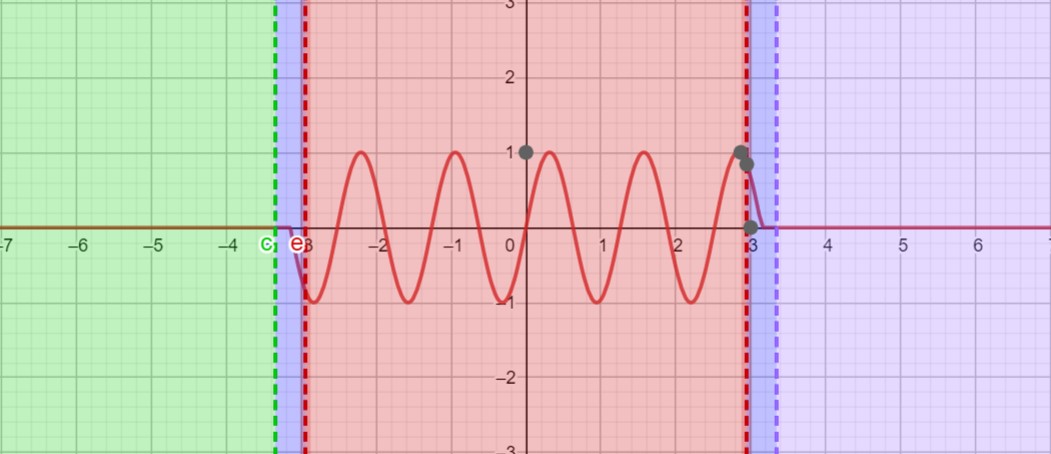}\end{center}
\caption{Intervals where we analyze $u_k(y,t)$.}\label{fig:regions}
\end{figure}

{\bf Case 1: $y<-\pi-\frac{1}{k}$.}

In this case, since $x\mapsto e^{-\frac{(y-x)^2}{4t}}$ is decreasing for $x\in[-\pi,\pi]$, the periodicity of $|\sin(kx)|$ implies that
for $i=-k+1,...,k-1$ we have
\begin{align*}
\int_{\pi \frac{i-1}{k}}^{\pi\frac{i}{k}} |\sin(kx)| \exp\left({-\frac{(y-x)^2}{4t}}\right)dx
\leq \int_{\pi \frac{i}{k}}^{\pi\frac{i+1}{k}} |\sin(kx)| \exp\left({-\frac{(y-x)^2}{4t}}\right)dx
\end{align*}

\begin{figure}[h!]
\begin{center}\includegraphics[scale=.4]{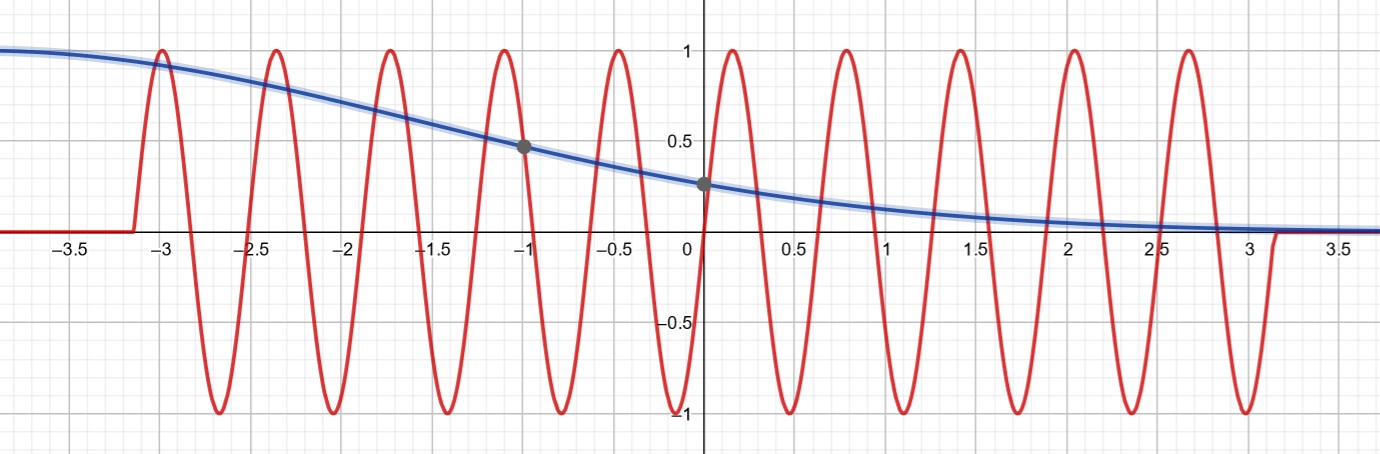}\end{center}
\caption{Case $y<-\pi-\frac{1}{k}$.}\label{fig:y_left_out}
\end{figure}

(see Figure \ref{fig:y_left_out}) and considering the alternating signs of $\sin(kx)$ we get
\begin{align*}
0\leq u_k(y,t) & \leq \frac{1}{\sqrt{4\pi t}}\int_{-\pi}^{-\pi\left(1+\frac{1}{m}\right)} \sin(kx) \exp\left({-\frac{(y-x)^2}{4t}}\right)dx\\ 
&\leq \frac{1}{k\sqrt{4\pi t}}\exp\left({-\frac{(y+\pi)^2}{4t}}\right).
\end{align*}

{\bf Case 2: $y >\pi+\frac{1}{k}$}

Analogously to case 1, we get that 
\begin{align*}
0\geq u_k(y,t) &\geq \frac{1}{\sqrt{4\pi t}}\int_{-\pi}^{-\pi\left(1+\frac{1}{k}\right)} \sin(kx) \exp\left({-\frac{(y-x)^2}{4t}}\right)dx\\ &\geq - \frac{1}{k\sqrt{4\pi t}}\exp\left({-\frac{(y-\pi)^2}{4t}}\right).
\end{align*}

{\bf Case 3: $y\in [-\pi-\frac{1}{k},-\pi+\frac{1}{k}]\cup[\pi-\frac{1}{k},\pi+\frac{1}{k}]$}.

Since $\|u_k(\cdot,t)\|_\infty\leq \|g_k\|_\infty$ for all $t>0$, then $|u_k(y,t)|\leq 1$ uniformly in these intervals.

{\bf Case 4: $y\in [-\pi+\frac{1}{k},\pi-\frac{1}{k}]$}.

For $y$ in this interval we observe that, similarly to case 1 and 2, and since the following integrals are absolutely convergent, we have that
\begin{align*}
0 \geq \frac{1}{\sqrt{4\pi t}}\int_{-\infty}^{-\pi} \sin(kx)\exp\left(-\frac{(y-x)^2}{4t}\right)dx \geq - \frac{1}{k\sqrt{4\pi t}}\exp\left(-\frac{(y+\pi)^2}{4t}\right),\\[.2em]
0 \leq \frac{1}{\sqrt{4\pi t}}\int_{\pi}^{\infty} \sin(kx)\exp\left(-\frac{(y-x)^2}{4t}\right)dx \leq \frac{1}{k\sqrt{4\pi t}}\exp\left(-\frac{(y-\pi)^2}{4t}\right),
\end{align*}
and at the same time, due to standard Fourier transform results, we have
\begin{align*}
\int_{-\infty}^\infty \sin(kx)\frac{1}{\sqrt{4\pi t}}\exp\left(-\frac{(x-y)^2}{4t}\right) dx = \sin(ky)\exp(-k^2t).
\end{align*}
In summary, for $y\in [-\pi+\frac{1}{k},\pi-\frac{1}{k}]$ we conclude that
%\begin{align*}
%\sin(ky)\exp(-k^2t) - \frac{1}{k\sqrt{4\pi t}}\exp\left(-\frac{(y+\pi)^2}{4t}\right) \leq u_k(y,t) \leq \sin(ky)\exp(-k^2t) + \frac{1}{k\sqrt{4\pi t}}\exp\left(-\frac{(y-\pi)^2}{4t}\right).
%\end{align*}
\begin{align*}
\sin(ky)\exp(-k^2t) - \frac{1}{k\sqrt{4\pi t}}\exp\left(-\frac{(y+\pi)^2}{4t}\right) \leq u_k(y,t),
\end{align*}
and
\begin{align*}
u_k(y,t) \leq \sin(ky)\exp(-k^2t) + \frac{1}{k\sqrt{4\pi t}}\exp\left(-\frac{(y-\pi)^2}{4t}\right).
\end{align*}

We proceed to compute the $L^2(\RR)$ norm of $u_k(\cdot,t)$ as
\begin{align*}
\|u_k(\cdot,t)\|^2_{L^2(\RR)}=& \int_{-\infty}^{-\pi-1/k} |u_k(y,t)|^2 dy +
\int_{-\pi-1/k}^{-\pi+1/k} |u_k(y,t)|^2 dy +\\
&\int_{-\pi+1/k}^{\pi-1/k} |u_k(y,t)|^2 dy +
\int_{\pi-1/k}^{\pi+1/k} |u_k(y,t)|^2 dy + \\
& \int_{\pi+1/k}^\infty |u_k(y,t)|^2 dy.
\end{align*}
We use that
\begin{align*}
\int_0^\infty \left[\frac{1}{k\sqrt{4\pi t}} \exp\left(-\frac{(y+\frac{1}{k})^2}{4t}\right)\right]^2 dy\leq \frac{1}{k^2}\frac{\exp\left(-\frac{1}{k^2t}\right)}{2\sqrt{4\pi t}}\leq  \frac{1}{k}\frac{e^{-1/2}}{4\sqrt{2\pi}}=:\frac{C}{k}
\end{align*}
to obtain the estimate
\begin{align*}
\|u_k(\cdot,t)\|^2_{L^2(\RR)}\leq & \frac{C}{k} + \frac{2}{k}+
\int_{-\pi-1/k}^{-\pi+1/k} |u_k(y,t)|^2 dy +
\frac{2}{k} + \frac{C}{k} \\
\leq &  \frac{C}{k} + \frac{2}{k}+\left(\pi e^{-k^2t}+\sqrt{\frac{C}{k}}\right)^2
+\frac{2}{k} + \frac{C}{k} \\
\leq & \tilde{C}\left(e^{-k^2t}+\frac{1}{k}\right).
\end{align*}

\subsection{Proof of main results}

\subsubsection{Proof of Theorem \ref{th:Lipschitz stability2}}
Based on semi-group theory, it follows that for all $u_0\in L^2(\RR^n)$ and $u$ the solution to \eqref{eq:Heat Equation},
\begin{equation}\label{SG_heat2}
u_0(x) = u(x,t) -\Delta \int^{t}_{0} u(x,s)ds,
\end{equation}
for all $t>0$ and a.e. $x\in\RR^n$ (see, for instance, \cite{Pazy}). %We are also assuming that $\supp{(u_0)}\subset B$.

For any test function $\varphi\in C^\infty_c(B)$ and $t>0$,
\[
\begin{aligned}
\int_B u_0(x)\varphi(x)dx &= \int_B u(x,t)\varphi(x)dx - \int_B \left(\int^t_0u(x,s)ds\right) \Delta\varphi(x)dx\\
&=\int_B u(x,t)\varphi(x)dx -  \int^t_0\int_Bu(x,s)\Delta\varphi(x)dxds,
\end{aligned}
\] 
therefore, by Cauchy-Schwarz inequality
\[
\left|\int_B u_0(x)\varphi(x)dx \right|\leq \|\varphi\|_{L^2(B)} \|u(\cdot,t)\|_{L^2(B)}+\|\varphi\|_{H^2_0(B)}\int^t_0\|u(\cdot,s)\|_{L^2(B)}ds. 
\]
%Here we understand $H^2_0(B)$ as the completion of smooth functions compactly supported inside $B$, with respect to the norm $\|\varphi\|_{H^2_0(B)} = \|\Delta\varphi\|_{L^2(B)}$. 
Integrating in time it directly follows the Lipschitz inequality
\[
\|u_0\|_{H^{-2}(B)}\leq T^{-1}(C_P+T)\|u\|_{L^1((0,T);L^2(B))},
\]
with $C_P$ the Poincare's inequality constant. We finally use Cauchy-Schwarz inequality one more time to conclude the proof.

\subsubsection{Proof of Theorem \ref{th:Lipschitz stability unbounded domain2}}

The sought inequality follows from the estimate
\begin{equation}\label{ineq_B_omega}
\|u\|_{L^2(B\times(0,T))}\leq C\|u\|_{L^2(\omega\times(0,T))},
\end{equation}
which is obtained from computations performed in the proof of Theorem 2.1, in the original version \cite{arratia2021}. Namely, the estimates (2.5) and (2.7) in the paper (with $\tau=0$) give
\[
\int_0^T\int_{\Theta}\frac{\hat{s}}{t(T-t)}|\theta|^2e^{2\hat{s}\hat{\zeta}}dxdt\leq C\|u\|^2_{L^2(\omega\times(0,T))}
\]
where we have adopted the same notation as in \cite{arratia2021}. We recall that $\theta=\rho u$ with $\rho$ a smooth cut-off supported in $\Theta$, a neighborhood of the compact region $\RR^n\backslash\omega$.

Since the weight $\frac{\hat{s}}{t(T-t)}$ is bounded from below by a positive constant independent of $T$, and $\rho = 1$ in $B\subset \RR^n\backslash\omega$, then simple estimations give \eqref{ineq_B_omega}.

\subsubsection{Proof of Theorem \ref{th:Curve stability}}
This proof is almost identical to the one that appears in \cite{arratia2021}. It is a direct consequence of Theorem 4.1 in the article, coupled with either Theorem 1.1 from \cite{arratia2021} or Theorem \ref{th:Lipschitz stability unbounded domain} from this corrigendum.

\subsubsection{Proof of Theorem \ref{th:LSFM stability}}
The proof remains the same, although we use Theorem 3 from this corrigendum which supersedes Theorem 1.4 from \cite{arratia2021}.

\subsection{Comments on the numerical experiments}

In the original paper \cite{arratia2021}, the section on numerical experiments studies the stability of the discretized inverse problem of LSFM by looking at the condition number of the resulting finite-dimensional linear problem.

These experiments provide important information by analyzing the dependence of the inversion stability on different parameters of the problem, such as the support of the target function, or if the set of measurements is complete or truncated, showing how reconstruction is considerably improved in regions illuminated directly by the laser (instead of just being illuminated by the tail of the gaussian profile).

In view of this corrigendum and the fact that there is no (strong) Lipschitz stability for the continuous LSFM inverse problem, it is incorrect to try to relate the condition number of the discretized problem to a theoretical Lipschitz stability constant, as it is done in the original paper. Nevertheless, we believe the numerical experiments are still an important complement to the analysis of the LSFM inverse problem, in particular, they provide relevant insight about the stability of this microscopy technique. An example of this corresponds to the large values of the condition number and how quickly the problem becomes more unstable under incomplete measurements, which it is in agreement with the fact that the theoretical analysis only guarantees a logarithmic stability.

%\section{An example appendix} 
%\lipsum[71]
%
%\begin{lemma}
%Test Lemma.
%\end{lemma}
%\end{comment}

\section*{Acknowledgments}

P.A. was partially funded by Basal Program CMM-AFB 170001 and Anid-Fondecyt grant \#1191903. P.A. also thanks the Department of Mathematical Engineering at Universidad de Chile.\newline
A.O. was partially funded by ANID-Fondecyt grants \#1191903 and \#1201311, Basal Program CMM-AFB 170001, FONDAP/15110009 and Millennium Nucleus for Cardiovascular Magnetic Resonance.\newline
B.P. would like to thanks the Department of Mathematics at Pontificia Universidad Cat\'olica de Chile, for their hospitality during a visit in December 2018-January 2019.\newline 
M.C. was partially funded by Anid-Fondecyt grant \#1191903.

\bibliographystyle{siamplain}
\bibliography{references}
\end{document}

%% file: ex_shared.tex
\usepackage{lipsum}
\usepackage{amsfonts}
\usepackage{graphicx}
\usepackage{epstopdf}
\usepackage{algorithmic}
\ifpdf
  \DeclareGraphicsExtensions{.eps,.pdf,.png,.jpg}
\else
  \DeclareGraphicsExtensions{.eps}
\fi

\newtheorem{Example}{Example}

\newtheorem{innercustomthm}{Theorem}
\newenvironment{customthm}[1]
  {\renewcommand\theinnercustomthm{#1}\innercustomthm}
  {\endinnercustomthm}

\newsiamremark{remark}{Remark}
\newsiamremark{hypothesis}{Hypothesis}
\crefname{hypothesis}{Hypothesis}{Hypotheses}
\newsiamthm{claim}{Claim}

\headers{Lipschitz stability for Backward Heat Equation}{P. Arratia, M. Courdurier, E. Cueva, A. Osses and B. Palacios}

\title{Lipschitz stability for Backward Heat Equation with application to Fluorescence Microscopy
\thanks{
This work was funded by ANID-Fondecyt grants \#1191903 and \#1201311, Basal Program CMM-AFB 170001, FONDAP/15110009.
}
}

\author{Pablo Arratia\thanks{Departamento de Ingeniería Matemática, FCFM, Universidad de Chile, Chile
  (\email{parratia@dim.uchile.cl}).}%, \url{http://www.imag.com/\string~ddoe/}).}
  \and Matías Courdurier\thanks{Facultad de Ciencias Matemáticas, Pontificia Universidad Católica de Chile, Chile
  (\email{mcourdurier@mat.uc.cl}).}
\and Evelyn Cueva\thanks{Research Center on Mathematical Modeling (MODEMAT), Escuela Polit\'ecnica Nacional, Quito, Ecuador
  (\email{ecueva@dim.uchile.cl}).}
\and Axel Osses\thanks{Departamento de Ingeniería Matemática and Centro de Modelamiento Matemático, UMI CNRS 2807, FCFM, Universidad de Chile, Chile
  (\email{axosses@dim.uchile.cl}).}
\and Benjamín Palacios\thanks{Department of Statistics, University of Chicago, US (\email{bpalacios@uchicago.edu}).}
}

\usepackage{amsopn}

\usepackage{amsmath}
\usepackage{amssymb}
\usepackage{bbm}
\usepackage{bm}
\usepackage{verbatim}
\usepackage{enumerate}

\DeclareMathOperator\supp{supp}

\newcommand{\R}{\mathbb{R}}
\newcommand{\RR}{\mathbb{R}}

\newcommand{\bb}{\boldsymbol{b}}
\newcommand{\bfmu}{\boldsymbol{\mu}}

\usepackage{accents}
\DeclareRobustCommand{\ubar}[1]{\underaccent{\bar}{#1}}
\newcommand{\dint}{\displaystyle\int}

\usepackage{placeins}
\makeatletter
\AtBeginDocument{%
  \expandafter\renewcommand\expandafter\subsection\expandafter
    {\expandafter\@fb@secFB\subsection}%
  \newcommand\@fb@secFB{\FloatBarrier
    \gdef\@fb@afterHHook{\@fb@topbarrier \gdef\@fb@afterHHook{}}}%
  \g@addto@macro\@afterheading{\@fb@afterHHook}%
  \gdef\@fb@afterHHook{}%
}
\makeatother

\newlength{\tempdima}
\newcommand{\rowname}[1]% #1 = text
{\rotatebox{90}{\makebox[\tempdima][c]{\textbf{#1}}}}

\makeatletter
\g@addto@macro\bfseries{\boldmath}
\makeatother